\documentclass[A4]{amsart}
\input xypic
\input xy \xyoption{all}

\let\originallhook\lhook
\let\originalrhook\rhook
\usepackage{MnSymbol}
\let\lhook\originallhook
\let\rhook\originalrhook

\usepackage{hyperref}
\usepackage{bbm}
\usepackage{amsmath}
\usepackage{graphicx}
\usepackage{caption}
\usepackage{subcaption}

\newtheorem{theorem}{Theorem}[section]
\newtheorem{proposition}[theorem]{Proposition}
\newtheorem{lemma}[theorem]{Lemma}
\newtheorem{corollary}[theorem]{Corollary}

\theoremstyle{definition}
\newtheorem{definition}[theorem]{Definition}
\newtheorem{bigremark}[theorem]{Remark}

\newtheorem{example}[theorem]{Example}

\newcounter{bean}

\newcommand{\seqm}[3]{\ensuremath{#1\stackrel{#2}
 {\longrightarrow}#3}}
\newcommand{\seqmm}[5]{\ensuremath{#1\stackrel{#2}
 {\longrightarrow}#3\stackrel{#4}{\longrightarrow}#5}}
\newcommand{\seqmmm}[7]{\ensuremath{#1\stackrel{#2}
 {\longrightarrow}#3\stackrel{#4}{\longrightarrow}#5
  \stackrel{#6}{\longrightarrow}#7}}

\newcommand{\br}[1]{\ensuremath{\left\{ #1 \right\}}}
\newcommand{\vbr}[1]{\ensuremath{\left\langle#1\right\rangle}}

\newcommand{\cvee}[3]{\displaystyle\bigvee^{#2}_{#1}#3}
\newcommand{\cprod}[3]{\displaystyle\prod^{#2}_{#1}#3}

\newcommand{\csum}[3]{\displaystyle\sum^{#2}_{#1}#3}

\newcommand{\cset}[2]{\br{#1\,\,\middle\vert\,\,#2}}
\newcommand{\cgen}[2]{\vbr{#1\,\,\middle\vert\,\,#2}}

\newcommand{\mc}[1]{\ensuremath{\mathcal{#1}}}
\newcommand{\mb}[1]{\ensuremath{\mathbb{#1}}}
\newcommand{\mf}[1]{\ensuremath{\mathfrak{#1}}}

\newcommand{\bd}{\ensuremath{\partial}}
\newcommand{\sm}{\setminus}

\newcommand{\wcolon}{\ensuremath{\,\colon\,}}

\DeclareMathOperator{\Int}{int}

\begin{document}

\title{Topology of Frame Field Design for Hex Meshing}
\author{Piotr Beben}
\address{ITI - International Technegroup Ltd., 4 Carisbrooke Court, 
Anderson Road, Swavesey CB24 4UQ, United Kingdom} 
\email{pdbcas2@gmail.com}


\begin{abstract} 
In the past decade frame fields have emerged as a promising approach for generating hexahedral meshes for CFD and CAE 
applications. One important problem asks for construction of a boundary-aligned frame field with prescribed singularity constraints 
over a volume that corresponds to a valid hexahedral mesh. We give a necessary and sufficient condition in terms of solutions to a 
system of monomial equations with variables in the binary octahedral group when a boundary frame field and singularity graph have 
been fixed. This is phrased with respect to general $CW$- decompositions of the volume, which allows some flexibility in simplifying 
these systems. Along the way we look at aspects of frame field design from an algebraic topological perspective, proving various results, 
some known, some new.
\end{abstract}

\maketitle

\section{Introduction}

Two and three dimensional meshes are used widely in engineering, mathematical physics, and computer science.
In computer graphics, CAD, CFD, and CAE they are used for representing geometry, or for solving problems related to heat transfer, 
fluid flow, and structural analysis via the finite element method. Built up as combinations of \emph{elements} such as triangles and 
quads, or tetrahedrons (tets), prisms, and hexahedrons (hexes), meshes subdivide geometry of interest such as the surface of an 
aircraft together with the air volume surrounding it, allowing discretization of numerical problems, for example modelling the airflow 
around an aircraft.

Depending on the application and context, there are various quality criteria that affect the speed and accuracy of simulations.
Quads are often preferred over triangles, and hexes preferred over prisms and tets. Distortion of elements is kept to a minimum, 
while orthogonal alignment with with the boundary is ensured. These criteria are sometimes taken relative to a metric sizing field,
which specifies alignment, size, and shape constraints on elements within a volume. For example, regions over an aircraft wing 
might prefer smaller elements that are elongated in the direction of airflow. Since a large number of smaller elements is expensive, 
larger elements sizes are often prefered where less accuracy is needed.
 
A completely automatic push-button 3D meshing algorithm that is able to meet any of these criteria has been considered a holy-grail 
of of finite element analysis for several decades. Though reliable algorithms for generating tetrahedral meshes exist, generating quality 
hex meshes with their inherent performance and accuracy advantages over tet meshes remains beyond reach. Alternatives 
to a good meshing algorithm may involve spending a great deal of effort manually constructing or modifying meshes that consist of 
millions of elements, settling for poor meshes that yield substandard output and long simulation times, or resorting to expensive tests 
on physical models. One poignant example is fan blade failure tests on jet engine turbines, costing millions of dollars for each design 
tested.

\begin{figure}[h]
    \centering
    \begin{subfigure}[h]{0.4\textwidth}
        \includegraphics[width=\textwidth]{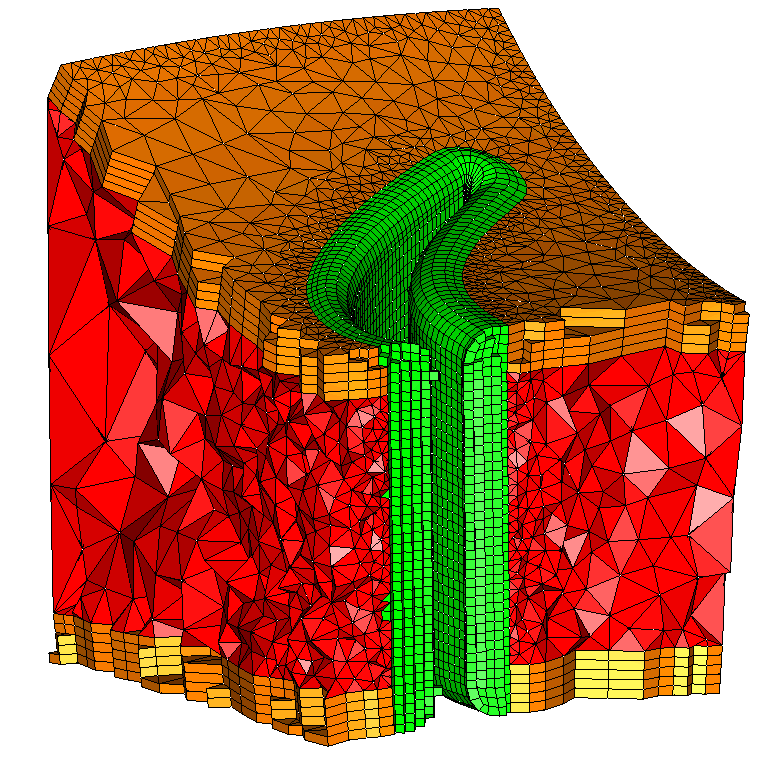}
    \end{subfigure}
    \begin{subfigure}[h]{0.4\textwidth}
        \includegraphics[width=\textwidth]{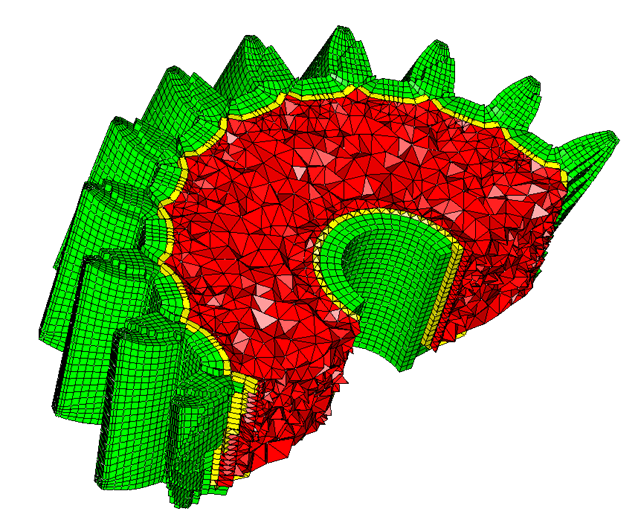}
    \end{subfigure}
    \caption{A $3$-dimensional boundary-aligned meshing of a gear and an air volume around a fan blade using hexes, prisms, and tets. 
	Smaller higher quality elements (hexes) are found near the boundary where more accuracy is needed.}\label{LF5}
\end{figure}

\subsection{Frame fields and meshing}

A $2$-\emph{frame field} (or \emph{cross field}) is a continuous assignment to a surface $M\sm\mc S$ (outside a set of singularities $\mc S$)
a field of orthgonal $2$-frames (crosses) where both axes are considered indistinguishable. Thus rotations of crosses by $90$ degrees return 
to the same cross. Similarly, in dimension three a $3$-frame field is a continuous assignment to a volume $M\sm\mc S$ a field of orthogonal 
$3$-frames where each of the $3$ orthogonal axes are indistinguishable. In both dimensions they are typically constrained to be aligned orthogonal
to the boundary of the volume or surface. Singularities for boundary aligned $2$-frame fields can be taken as isolated points, but this is not the case 
for $3$-frame fields since they sometimes (necessarily) form continuous subspaces homeomorphic to graphs.

The connection of frame fields to quad/hex meshing can be seen by filling a volume with individual elements. Large elements usually 
fill a volume poorly, leaving behind large irregular gaps that cannot be filled. Mitigating this by choosing smaller elements does not solve 
our problem since a mesh with so many elements is wasteful, if not infeasible. But it does lead to an insight: where corners of quads 
or hexes meet more-or-less orthogonally, we can assign frames or crosses or frames by alligning with their boundaries. The result
is a discrete field of frames over our body $M$ that approaches a continuous frame field on $M$ as elements are shrunk -- defined 
outside singularities where the orthogonal constraints are not broken. In this way, an infinitesimally fine mesh can be regarded as a continuous 
frame field.

Once a continuous frame field has been defined, there are several methods for reversing this process to obtain a finite boundary-aligned 
quad or hex meshing. For example, the CubeCover algorithm~\cite{Nieser2011} computes a volume parameterization that best agrees with a discrete 
boundary-aligned frame field generated over the volume, from which a hex mesh can be obtained by extracting iso-surfaces of the parameterization 
to form the faces of every hex. Another approach known as integer-grid maps~\cite{Liu2018} yields a hex mesh directly by pulling 
back a hexahedral grid back to align with a frame field. In either case a frame field should first be constructed that is compatible with a hex-only mesh, 
meaning a subspace of singularities must be carefully chosen together with frames appropriately aligned on the boundary and around
the singularities to yield a valid hex mesh. Both these approaches extend 2D quad meshing techniques such as QuadCover~\cite{Nieser2007, Bommes2013}. 
A somewhat different approach~\cite{Fogg2015, Myles2014, Kowalski2014} involves tracing cross field aligned lines from singularities and boundary 
corners to produce the perimeters of boundary-aligned quads subdividing a surface, each of which can easily be meshed. 

This is not an exhaustive list of techniques; other methods for extracting meshes from frame fields can be found in~\cite{Tarini2011, 
Lyon2016, Gao2017, Papadimitrakis, Pietroni2019}. The computational probem of generating frame fields suitable for hex meshing is also 
well-studied~\cite{
Crane2010, Huang2011, Huang2012, Li2012, Ray2016, Solomon2017, Viertel2018, Chemin2019, Crane2019, Golovaty2020, Palmer2020}, 
while theoretical and mathematical aspects (geometrical and topological, differential and algebraic) have been studied in~\cite{
Bunin2007, Ray2008, Nieser2011, Li2012, Fogg2017, Beaufort2017, Liu2018, Crane2019, Palmer2020}. We will mention any overlap 
with these works as we go along.

\begin{figure}[h]
    \centering
    \begin{subfigure}[h]{0.55\textwidth}
        \includegraphics[width=\textwidth]{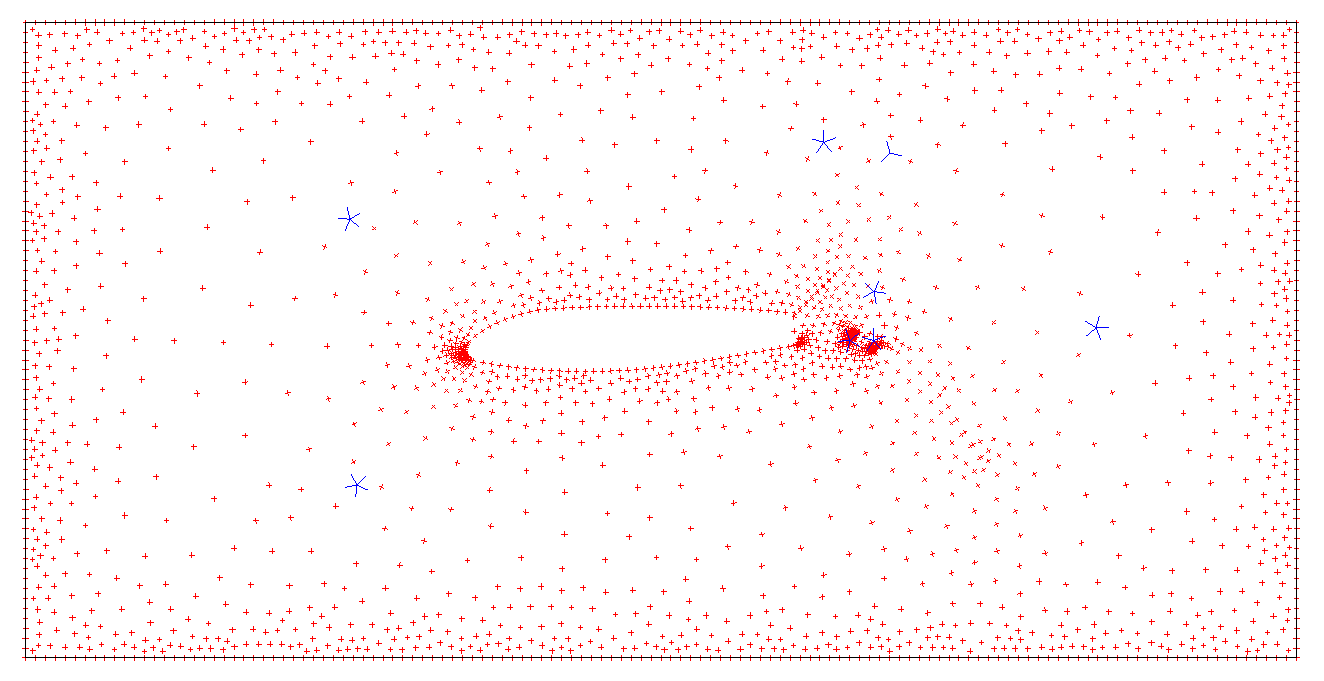}
    \end{subfigure}
    \begin{subfigure}[h]{0.55\textwidth}
        \includegraphics[width=\textwidth]{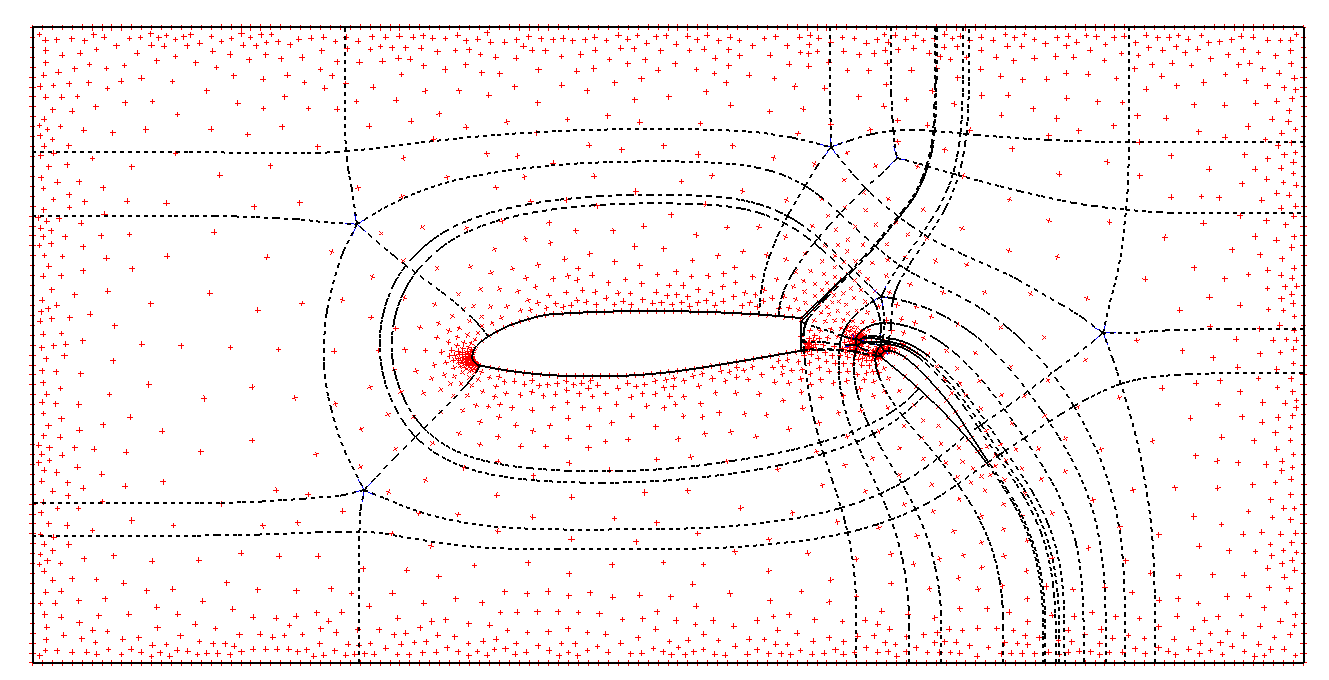}
    \end{subfigure}
    \begin{subfigure}[h]{0.55\textwidth}
        \includegraphics[width=\textwidth]{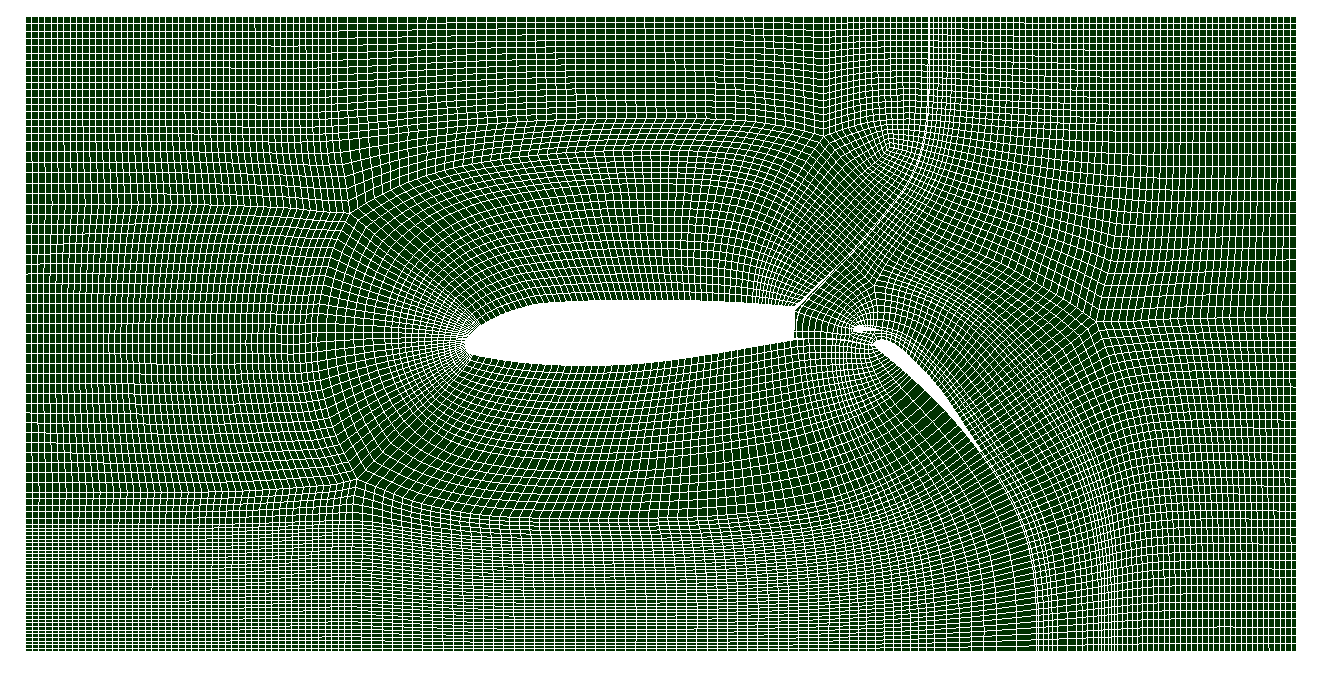}
    \end{subfigure}
    \caption{A 2D cross section of air volume around an aircraft wing and flap with a boundary aligned 
		discretized cross field (red) with singularities (blue) defined over a triangulation (invisible). 
		\emph{Separatrix} lines are traced following the cross field from singularities and corners 
		depending on their \emph{index} to give a quad-dominant partition of the surface. 
		A finer 2D quad meshing is obtained by meshing each quad region in the partition, after
                 balancing line divisions across shared lines.} 
\end{figure}

\begin{figure}[h]
    \centering
    \includegraphics[width=0.65\textwidth]{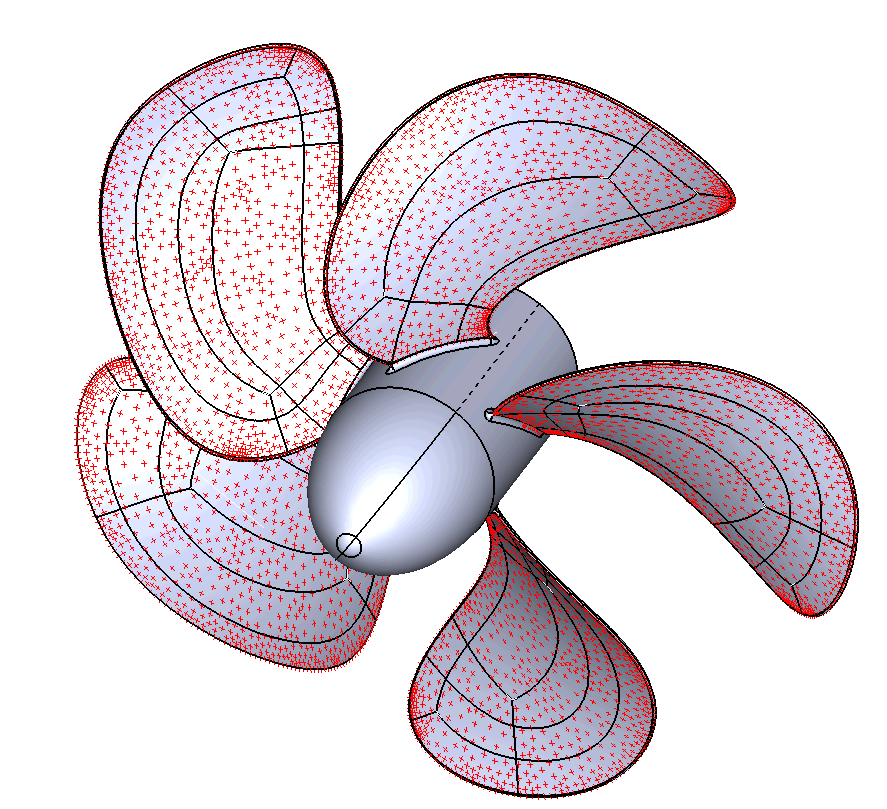}
    \caption{A cross field-generated subdivision on a series of propeller blades. Any boundary aligned 
                  $3$-frame fiield generated over the volume of the propeller restricts to a cross field ($2$-frame field)
                  on the boundary like the one present here by forgetting axes aligned with surface normals.}
\end{figure}

\begin{figure}[h]
    \centering
    \begin{subfigure}[h]{0.4\textwidth}
        \includegraphics[width=\textwidth]{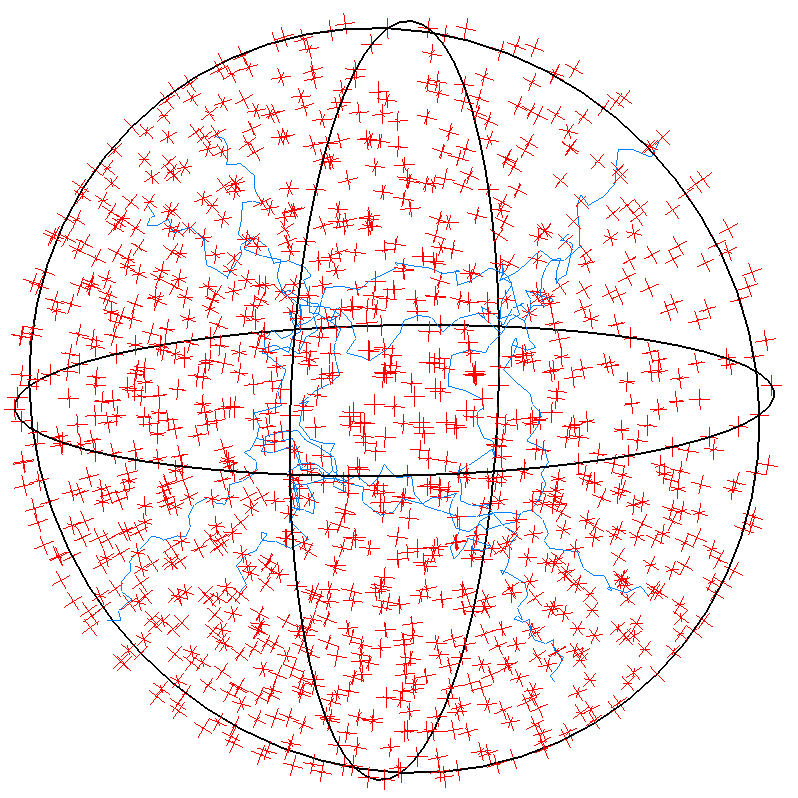}
    \end{subfigure}
    \begin{subfigure}[h]{0.4\textwidth}
        \includegraphics[width=\textwidth]{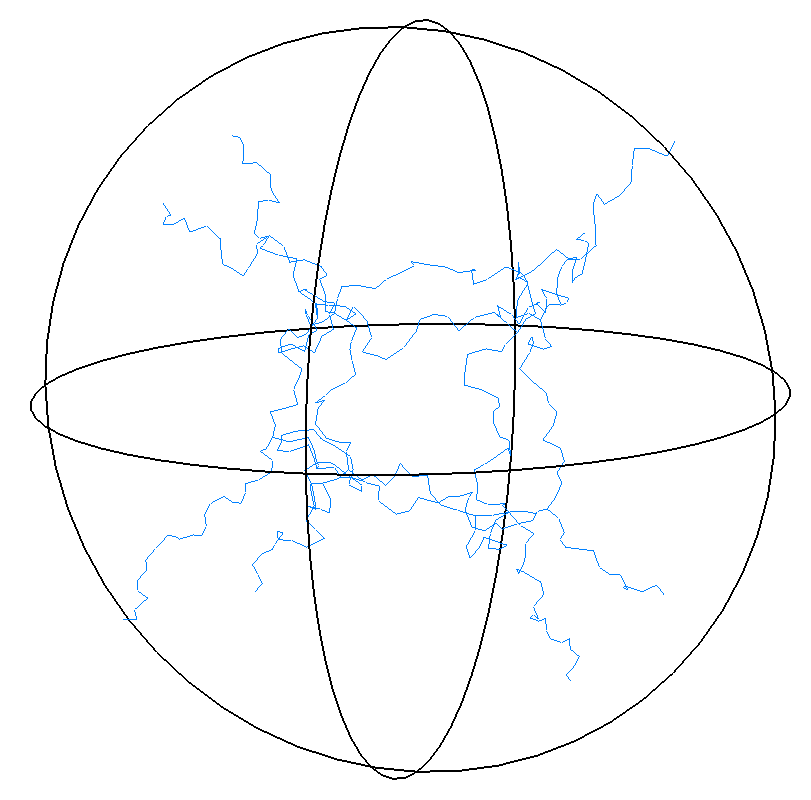}
    \end{subfigure}
    \caption{A boundary-aligned $3$-frame field and its singularity graph generated piecewise-linearly over a 
	       $3$-ball by interpolating over a tet subdivision.}
\end{figure}

\section{Abstraction of Frame Fields}

\subsection{General definitions}

Let $M$ be a Riemannian $n$-manifold, $TM$ and $SM$ be the tangent and unit tangent bundles of $M$, 
$TM_x\cong \mb R^n$ and $SM_x\cong S^{n-1}$ be their restrictions at a point $x\in M$. 
Consider the \emph{orthogonal frame bundle} 
$$
\rho\wcolon\seqm{V_m(M)}{}{M} 
$$
induced by the tangent space $TM$. 
Here $V_m(M)$ consists is the subspace of direct sum of bundles $SM^{\oplus m}$, 
consisting for each $x\in M$ of $(v_1,\ldots,v_m)$ such that the $v_i$ are mutually orthogonal in $SM_x$. 
Each fiber $\rho^{-1}(x)$ is the Stiefel manifold $\mc V_{n,m}$ of orthonormal $m$-frames in $\mb R^n$
(orthonormal $(n\times m)$-matrices, the column vectors forming the directed axis of the frame).

Let $\mc B_m$ denote the \emph{full hyperoctahedral group} consisting of symmetries of an $m$-dimensional hypercube. 
As a matrix group, $\mc B_m$ is the subgroup of the orthogonal group $O(m)=\mc V_{m,m}$ consisting of the 
$2^m m!$ different matrices obtained by permuting and reflecting columns in the $(m\times m)$-identity matrix.
The right $\mc B_m$ action on $\mc V_{n,m}$ by right matrix multiplication is the free action that permutes and reflects 
column vectors, and extends to a fiber preserving action on the bundle $V_m(M)$. 
Taking orbit spaces, $\rho$ quotients onto a $\mc V_{n,m}/\mc B_m$-bundle 
$$
\rho'\wcolon\seqm{V_m(M)/\mc B_m}{}{M}.
$$
Notice the orbit space $V_m(M)/\mc B_m$ is the quotient $V_m(M)/\sim$ under the identifications 
$$
(v_1,\ldots,v_m)\sim (\lambda_1 v_{\sigma(1)},\ldots,\lambda_m v_{\sigma(m)})
$$
for each permutation $\sigma$ in the symmetric group $S_m$ and $\lambda_i\in \mb R\sm\{0\}$. 
We may think of $V_m(M)/\mc B_m$ ia the space of frames consisting of $m$ orthogonal axis in which 
there is no sense of direction and ordering (labelling) of the axis. Thus, any reflection of a frame along an axis 
leads back to the same frame, as does any rotation that has the same image.  

\begin{definition}
A \emph{directed} $m$-\emph{frame} \emph{field} on $M$ with set of \emph{singularities} $\mc S\subseteq M$ is a 
continuous section \seqm{M\sm\mc S}{s}{V_m(M)} so that \seqmm{M\sm\mc S}{s}{V_m(M)}{\rho}{M} is the inclusion. 
\end{definition}

We say it is \emph{nowhere vanishing} if $\mc S=\emptyset$. 
\emph{Isolated singularities} in $\mc S$ are the dimension $0$ singularities that are not limit points
of other singularities. 

\begin{definition}
Similarly, an (\emph{undirected}) $m$-\emph{frame} \emph{field} is a continuous section 
\seqm{M\sm\mc S}{s}{V_m(M)/\mc B_m} so that \seqmm{M\sm\mc S}{s}{V_m(M)/\mc B_m}{\rho'}{M} is the inclusion. 
\end{definition}

Every directed frame field induces an undirected one by forgetting ordering and direction of axes.
Our focus is one undirected frame fields, which we refer to simply as \emph{frame fields}.
Undirected frame fields (as opposed to directed) are used for the purpose of meshing 
since there is generally no sense of direction in any hex or quad subdivision.  
Moreover, an directed frame field induces and undirected one by forgetting ordering and direction, 
but the reverse is not true. In this way undirected frame fields are more flexible.

\subsection{Over trivial tangent bundles}

Suppose the tangent bundle $TM$ restricts to a trivial sub-bundle $T(M\sm\mc S)\cong (M\sm\mc S)\cong \mb R^n$ 
over $M\sm\mc S$. This is equivalent to $M\sm\mc S$ having $n$ linearly independent vector fields (being \emph{parallelizable}).
In this case the trivialization defines trivializations of the sub-bundles
$$
V_m(M\sm\mc S)\,\cong\, (M\sm\mc S)\times \mc V_{n,m}
$$
and
$$
V_m(M\sm\mc S)/\mc B_m\,\cong\, (M\sm\mc S)\times  \mc V_{n,m}/\mc B_m
$$
such that the projection maps $\rho$ and $\rho'$ restrict to the identity on the left factors $M\sm\mc S$.
Directed and undirected $m$-frame field can then be regarded as continuous maps 
$$
\seqm{M\sm\mc S}{}{ \mc V_{n,m}} 
$$
and 
$$
\seqm{M\sm\mc S}{}{ \mc V_{n,m}/\mc B_m}.
$$
A general setting where this happens (for $m=n$) is when $M$ is a compact $n$-manifold smoothly embedded 
in the same dimension $\mb R^n$; a trivialization of its tangent bundle is obtained from that of $\mb R^n$, 
and an $n$-frame field is simply a continuous map
$$
f\wcolon\seqm{M\sm\mc S}{}{O(n)/\mc B_n}.
$$

Write $O(n)$ as the disjoint union
$$
O(n)=SO(n)\sqcup SO^{-}(n)
$$ 
of subgroups of positive and negative determinant matrices.
Let $\mc D_n$ (the \emph{rotation octahedral group}) denote the subgroup of $\mc B_n$ consisting of only those matrices 
with positive determinant. Namely, this is the group of rotational symmetries of an orthonormal $n$-frame 
(or $n$-hypercube). The action of $\mc B_n$ on $O(n)$ restricts to an action of $\mc D_n$ on the subgroup $SO(n)$, 
and the inclusion \seqm{SO(n)}{}{O(n)} maps to a homeomorphism
\begin{equation}
\label{ESOn}
SO(n)/\mc D_n\,\cong\, O(n)/\mc B_n.
\end{equation}
Thus, we may regard the $n$-frame fields $f$ above as maps 
$$
f\colon\seqm{M\sm\mc S}{}{SO(n)/\mc D_n}.
$$
The standard embedding of $SO(n-1)$ as a subgroup of $SO(n)$ gives an embedding of $\mc D_{n-1}$ as a subgroup of $\mc D_n$.
Namely, $\mc D_{n-1}$ consists of those (positive determinant) $(n\times n)$-matrices in $\mc D_n$ whose first row vector 
and first column vector (directed axis) are both $(1,0,\ldots,0)$. 
Then $\mc D_{n-1}$ acts on $SO(n)$ by permuting and reflecting only the column vectors $v_i$ for $i>1$ of a frame, 
leaving the first fixed, and the covering map \seqm{SO(n)}{p}{SO(n)/\mc D_n} factors as
$$
p\wcolon\seqmm{SO(n)}{\tilde p}{SO(n)/\mc D_{n-1}}{p_2}{SO(n)/\mc D_n}.
$$
through the \emph{partial quotient} orbit space $SO(n)/\mc D_{n-1}$. The orbit space quotient maps $p$ and $\tilde p$
are covering maps since $\mc D_i$ is finite and acts freely on the total space.
The map $p_2$ sends a representative $(v_1,\ldots,v_m)$ to itself in $SO(n)/\mc D_n$, which is well defined since 
$\mc D_{n-1}$ as a subgroup of $\mc D_n$. We think of $p_2$ as forgetting the direction and ordering of the
distinguished axis $v_1$ in each frame, making it indistinguishable from the other axis.  To simplify notation, we will denote
$$
\widetilde{\mc O}_n\,\colon=\, SO(n)/\mc D_{n-1}.
$$
and
$$
\mc O_n \,\colon=\, SO(n)/\mc D_n \,\cong\, O(n)/\mc B_n
$$
for both $SO(n)/\mc D_n$ and $O(n)/\mc B_n$ by the above homeomorphism. 

A boundary-aligned $3$-frame field \seqm{M\sm\mc S}{}{\mc O_3} lifts to a map \seqm{\bd M\sm\mc S}{}{\widetilde{\mc O}_3}
on its boundary. More generally, a $2$-frame field on a smooth embedded surface in $\mb R^3$ can be regarded 
as such a map. 

In studying maps between topological spaces (such as $f$), it is often useful to compute topological invariants for the spaces 
in question. We start by looking at the homotopy and (co)homology of $SO(n)/\mc D_m$, mostly when $n=3$.
Since most of the spaces we deal with will be path connected, the choice of basepoint for homotopy groups is irrelevant, 
so we usually make no mention of it. There is no work to be done for the case $n=2$ since $O(2)/\mc B_2\cong S^1$.

\section{Topology of $SO(n)/\mc D_m$}

\subsection{Homotopy groups}

Take the universal cover \seqm{Spin(n)}{q}{SO(n)}. Its fiber is $\mb Z_2$;
when $n=3$, $Spin(3)\cong S^3$, and $q$ is homeomorphic to the universal double cover 
\seqm{S^3}{q}{SO(3)} mapping unit quaternions to 3D rotations
(or the standard universal double cover \seqm{S^3}{q}{\mb RP^3\cong SO(3)}). 
The preimage $q^{-1}(\mc D_n)\subseteq Spin(n)$ of the octahedral subgroup $\mc D_n\subseteq SO(n)$ 
is known as  the \emph{binary hyperoctahedral} group, denoted $2\mc D_n$. This exists in a (non-split) 
exact sequence  $$1\mapsto \mb Z_2\mapsto 2\mc D_n\mapsto \mc D_n\mapsto 1$$
with the abelianization $(2\mc D_n)_{ab}$ isomorphic to $\mb Z_2$. 
When $n=3$, $2\mc D_3$ is an order $48$ subgroup of $S^3$, with presentation 
$$
\cgen{r,s,t}{r^2=s^3=rst}
$$
with unit quaternion generators $r := \frac{1}{\sqrt{2}}(i+j)$, $s := \frac{1}{2}(1+i+j+k)$, 
$t := \frac{1}{\sqrt{2}}(1+i)$ in $S^3$ ($rst=-1$).

The orbit quotient map \seqm{SO(n)}{p}{\mc O_n} is a covering map since $\mc D_n$ is finite and 
acts freely on $SO(n)$. Likewise, finiteness of the fibers $p^{-1}(x)=\mc D_n$ implies
the composition $\seqm{Spin(n)}{p\circ q}{\mc O_n}$ of coverings $q$ and $p$ is a covering also.
This is a universal covering since $Spin(n)$ is simply connected, with fiber 
$q^{-1}\circ p^{-1}(x)=q^{-1}(\mc D_n) = 2\mc D_n$. Thus
\begin{equation}
\label{EPi1}
\pi_1(\mc O_n)\,\cong\, 2\mc D_n.
\end{equation}
(A similar computation of the fundamental group in the context of frame fields has been given recently 
in~\cite{Palmer2020}.)

As for higher homotopy groups, the covering $p$ is a fiber bundle since $\mc O_n$ is connected.
Then the homotopy long exact sequence for $p$ implies 
\begin{equation}
\label{EPin}
\pi_i(\mc O_n)\cong \pi_i(SO(n))\quad\mbox{for }i>1
\end{equation}
since the fibre $\mc D_n$ of $p$ is $0$-dimensional. In particular, since $\pi_2(SO(3))=0$~\cite{MimuraToda}, 
\begin{equation}
\label{EZero}
\pi_2(\mc O_3)=0
\end{equation}
The second homotopy group being zero will be crucial later (we could have also used the covering 
\seqm{S^3}{}{SO(3)} and $\pi_2(S^3)=0$).

\subsection{$\mb Z_2$-Cohomology}
We give a quick proof for the case $n=3$ using Poincar\'e duality and the above computation of the 
fundamental group. Write $H^i=H^2(\mc O_3;\mb Z_2)$ for the $\mb Z_2$-cohomology. 
Note that $\mc O_3$ is a $3$-manifold. Then $H^3\cong \mb Z_2$ and $H^2\cong H^1$. 
Moreover, for degree $1$ integral homology, the Hurewicz theorem implies 
$H_1(\mc O_3)\cong\mb Z_2$ since it is the abelianization of $\pi_1(\mc O_3)\cong 2\mc D_3$. 
Then $H^1\cong (2\mc D_3)\cong\mb Z_2$ by the universal coefficient theorem. 
The cup product structure can now easily be read off from Poincar\'e duality, and we obtain the ring
\begin{equation}
\label{ECohom}
H^*(\mc O_3;\mb Z_2)\,\cong\, \mb Z_2[\alpha]/\alpha^4
\end{equation}
where $|\alpha|=1$. This is the same cohomology ring as $H^*(SO(3);\mb Z_2)$. 
In fact, we see that \seqm{SO(3)}{q}{\mc O_3} induces a ring isomorphism
on $H^*(;\mb Z_2)$ since it is an isomorphism on $H^1(;\mb Z_2)$,
which follows by the Hurewicz theorem and homotopy long exact sequence for the covering $q$,
and because the abelianization of $\mb Z_2\mapsto 2\mc D_3$ is an automorphism of $\mb Z_2$.
An alternate computation agreeing with the one here can found in~\cite{Satoshi} (among other interesting 
orbit spaces considered there).

\subsection{Partial quotients}
Notice $\mc D_2$ acts fiber-wise on the standard $S^1=SO(2)$-bundle \seqm{SO(3)}{\rho}{S^2}, 
thus it quotients to an $S^1\cong \mc O_2$-bundle \seqm{\widetilde{\mc O}_3}{\tilde\rho}{S^2}.
Namely, we have a commutative diagram of fiber bundle sequences 
\begin{equation}
\label{EFib}
\diagram
S^1 \cong SO(2)\rto^{}\dto^{\tilde p} & SO(3) \rto^{\rho}\dto^{\tilde p} & S^2\ddouble\\
S^1 \cong \mc O_2\rto^{} & \widetilde{\mc O}_3 \rto^{\tilde\rho} & S^2,
\enddiagram
\end{equation}
where $\rho$ and $\tilde \rho$ are given on each frame by projecting onto the first (in each case distinguished) unit column vector,
and $\tilde p$ are the orbit space quotients, with the map of fibers \seqm{SO(2)}{\tilde p}{\mc O_2} homeomorphic to multiplication 
\seqm{S^1}{4}{S^1} by $4$. 
Since $\pi_1(SO(n))\cong \mb Z_2$, the boundary map \seqm{\pi_2(S^2)}{\bd}{\pi_1(S^1)}  
in the homotopy long exact sequence for the top fiber bundle is \seqm{\mb{Z}}{2}{\mb{Z}}. 
The morphism of homotopy long exact sequences coming from this diagram gives the 
boundary map  \seqm{\pi_2(S^2)}{\tilde\bd}{\pi_1(S^1)} for the bottom sequence as \seqm{\mb{Z}}{8}{\mb{Z}}, and so
\begin{equation}
\label{Ehomo}
\pi_1( \widetilde{\mc O}_3)\cong \mb Z_8.
\end{equation}
Lastly, 
\begin{equation}
\label{Ehomo2}
\pi_i(\widetilde{\mc O}_n)\cong \pi_i(SO(n))\mbox{ for }i>1 
\end{equation}
by the covering $\tilde p$.

Homology can also determined from this diagram. 
The Leray-Serre spectral sequence for the top bundle is well known, and easy to compute~\cite{MimuraToda}. 
Its first and only transgression is given by $\bd$ as multiplication by $2$. 
The Leray-Serre spectral sequence for the bottom bundle can then be computed from the top 
by spectral sequence comparison applied to diagram~\ref{EFib} and the boundary map $\tilde\bd$. 
From this, one obtains an isomorphism of integral homology
\begin{equation}
\label{EHom}
H_i(\widetilde{\mc O}_3) \,\cong\,
\begin{cases}
\mb Z & \mbox{if }i=0,3;\\
\mb Z_8 & \mbox{if }i=1;\\
0 & \mbox{otherwise},
\end{cases}
\end{equation}
with $(p_1)_*$ being a multiplication by $4$ on $H_i()$ for $i=3$ and $0$ otherwise for $i>0$.

\subsection{Back down to earth}

For practical applications of frame fields, the dimension $n$ is often $2$ and $3$, and one usually assumes the following

\begin{itemize}
\item \textbf{Smooth embedding:} $M$ is smoothly embedded in the same dimension $\mb R^n$. 
So directed and undirected $n$-frame fields are simply maps \seqm{M\sm\mc S}{}{O(n)} and \seqm{M\sm\mc S}{f}{\mc O_n}.
\item \textbf{Boundary alignment:} Frames are \emph{boundary-aligned} on $\bd M$ wherever they are defined. 
That is, if $M$ is a manifold without corners, one of the axes of each frame is normal to the tangent plane on the 
boundary where it is located. If there are corners, frames can interpolated and smoothed in their vicinity,
possibly breaking boundary alignment near them. For now we will assume there are no corners.
\end{itemize}

In particular, these conditions imply an $n$-frame field restricts on the boundary $\bd M$ to an $(n-1)$-frame field 
$$
\seqm{\bd M\sm\mc S}{}{\mc V_{n,n-1}/\mc B_{n-1}}
$$ 
by forgetting the axis normal to the boundary.

\subsection{Singularities}

The other aspect besides boundary alignment that affects the utility of frame fields is the pattern and quantity of singularities 
together with the alignment of frames around them. In a hex-only mesh they correspond to irregular regions of the mesh where 
hexes are distorted away from having orthogonal boundary faces, and where more or less than $4$ and $8$ hexes share edge 
and vertex corners, both of which can affect their performance in CAE applications. This is in the best case scenario however, 
since in the worst case there is no valid hex-only mesh corresponding to certain choices of singularities~\cite{Liu2018}, 
so they must be chosen carefully. For some very simple surfaces, any boundary-aligned $2$-frame field has at least one singularity:

\begin{example}
The $2$-disk $D^2\subseteq\mb R^2$ cannot have a boundary-aligned $2$-frame field 
\seqm{D^2}{f}{\mc O_2} with no singularities. 
Otherwise the composite $g\wcolon\seqmm{S^1}{\iota}{D^2}{f}{S^1}$, where $\iota$ is the
embedding of the boundary of $D^2$, is nullhomotopic since $D^2$ is contractible. 
But this is impossible since $g$ is a homeomorphism by boundary-alignment.
\end{example}

On more general compact surfaces there a guarantee that a boundary-aligned $2$-frame field can be taken with only 
isolated point singularities, of which there are finitely many. For example, by triangulating the surface, assigning 
boundary-aligned crosses to its vertices, then interpolating crosses onto the interior of edges, and finally interpolating 
onto triangle interiors by shrinking triangle boundaries to leave a singularity at the centroid. 
Moreover, if we don't place any constraints on the \emph{index} of singularities, then at most one point singularity 
is necessary (as we will see later). This is similar to the situation for boundary-aligned vector fields on smooth compact 
manifolds, where a vector field exists with a single isolated point singularity lying in the interior.
In contrast, there are manifolds on which boundary aligned $3$-frame fields must have continuous subspaces of singularities 
that are not even contractible:

\begin{example}
\label{ENoncontract}
Take $n=3$, $M$ to be the standard solid torus $D^2\times S^1$ embedded in $\mb R^3$,
Let \seqm{M\sm\mc S}{f}{\mc O_3} be any undirected boundary aligned $3$-frame field
with the following boundary constraint: the tangential axes of each frame on the boundary $\bd M=S^1\times S^1$ 
is each aligned with one of the coordinate circles $S^1$. 
Assume $\mc S$ in the interior of $M$, has countably many connected components. 
Then at least one connected component of $\mc S$ is not contractible as follows. 

Notice any embedding \seqm{S^1}{\iota}{M\sm\mc S} into a coordinate circle $S^1\times \{x\}\subseteq \bd M$
is not nullhomotopic. To see this, notice by our assumed frame alignment $f\circ\iota$ lifts through 
\seqm{SO(3)}{p}{\mc O_3} to a map \seqm{S^1}{\ell}{SO(3)} such that $\ell$ is the standard inclusion 
of a fiber $S^1$ in the fiber bundle \seqmm{S^1}{\ell}{SO(3)}{\rho}{S^2}.
But $\ell$ is not nullhomotopic since it induces a surjection onto $\pi_1(SO(3))\cong \mb Z_2$ by the 
homotopy long exact sequence for this bundle. Then $\iota=p\circ\ell$ also cannot be nullhomotopic since 
$[\ell]\in\pi_1(SO(3))$, and since $p$ as a covering map must induce an injection on $\pi_1()$ .
 
Now suppose all connected components of $\mc S$ are contractible. 
Then by Alexander duality for manifolds, $M\sm\mc S$ has the homology
of a bouquet of a circle and a countable wedge of $2$-spheres, 
and so $M\sm\mc S$ is homotopy equivalent to such a bouquet.
Moreover, the homology generator of the circle in the bouquet corresponds to the generator of the 
second coordinate circle in the product $M=D^2\times S^1$, and thus is mapped trivially by $\iota_*$.   
But in this case homology detects the homotopy class of the map 
\seqm{S^1}{\iota}{M\sm\mc S\simeq S^1\vee\bigvee S^2} into this bouquet, 
and so $\iota$ is nullhomotopic since it induces a trivial map on homology, a contradiction. 
Thus at least one component must not be contractible, and in particular, cannot be an isolated point.  
\end{example}

\begin{definition}
A \emph{singularity graph} $G$ of a frame field $f$ with singularity set $\mc S$ 
is a subspace $G\subseteq \mc S$ that is homeomorphic to a topological graph,
and is smoothly embedded in $M$ on the interiors of it edges.
The (closed) edges of the graph are then smooth submanifolds homeomorphic to $[0,1]$, 
containing vertices at endpoints $0$ and $1$, and \emph{open} and \emph{half-open} 
edges are respecitvely the subintervals $(0,1)$ and $(0,1]$ with both vertices removed
and at least one vertex removed. A \emph{branch} is an edge with at least one vertex 
that is a \emph{leaf} vertex (incident to only one edge), and all other vertices removed. 
Thus it is either closed or half open.
\end{definition}

If all connected subspaces $G\subseteq\mc S$ are  graphs, we refer to $\mc S$ itself as a singularity graph.
While boundary-aligned $2$-frame fields on surfaces have isolated point singularities,  
singularities on boundary aligned $3$-frame fields on $3$-manifolds are typically graphs.   
For example, if we triangulate a manifold and assigning boundary-aligned frames to vertices,
then interpolate from vertices to interiors of edges, then to interiors of triangles, and finally to interiors of each tetrahedron. 
The last step involves interpolating frames on each tetrahedron boundary into the interior.
This causes singularity lines to be traced out towards the centroid starting from point singularities on boundary triangles,
which results in a singularity graph. There are several additional nice properties in this construction: 
(1) there are no isolated point singularities; (2) there are no branches in the graph that are in the interior of $M$; 
(3) the interiors of all edges lie in the interior of $M$, with only endpoints lying on $\bd M$, all of which are leaf vertices;  
(4) thus the $3$-frame field restricts to a $2$-frame field on $\bd M$ with only isolated point singularities.

This still leaves open whether any $3$-frame field with singularity graph $\mc S$ can be modified locally
so as to satisfy properties (1)-(4). We mean this in the following sense.
Let respectively $\bd B$ and $\Int(B)=B\sm\bd B$ denote boundary and interior of $B\subseteq M$ as a subspace of 
$\mb R^n$, and $\bd_M B$ and $\Int_M(B)=B\sm\bd_M B$ the boundary and interior of $B$ as a subspace of $M$. 

\begin{definition}
Write $N=M\sm\mc S$. Define subspace of singularities $U\subseteq\mc S$ or an $n$-frame field \seqm{N}{f}{\mc O_n} 
to be \emph{redundant} if there exists a closed subspace $B\subseteq M$ and an $n$-frame field defined on $\Int_M(B)$
$$
f'\wcolon\seqm{(N\cup \Int_M(B))}{}{\mc O_n}
$$  
such that $U\subset \Int_M(B)$, and $f$ agrees with $f'$ on $N\sm \Int_M(B)$. 
If $f$ is boundary aligned, we require $f'$ to be boundary aligned as well.
\end{definition}

\begin{proposition}
\label{PRedundant}
Suppose a singularity subspace $P\subseteq\mc S$ is contained in $\Int(B)$ of a closed $n$-ball $B\subseteq M$, 
and $\bd B$ contains at most one point singularity of $\mc S$. Then $P$ is redundant.
\end{proposition}

\begin{proof}
Suppose $\bd B$ contains no singularities. Since $\pi_2(\mc O_3)\cong \pi_2(SO(3))=0$ and $\bd B\cong S^2$, 
there exists a nullhomotopy of any map \seqm{\bd B}{}{\mc O_3}. From this we can construct an interpolation 
$f'$ of $f$ from $\bd B$ onto all of $\Int(B)\cong \Int(D^3)\cong S^2\times(0,1)\cup \{s\}$ by applying the 
nullhomotopy over $S^2\times\{t\}$ for each $t\in(0,1)$. 

Now suppose $\bd B$ contains a single point singularity $s$. Note $B\cong ([0,1]\times \bd B)/\sim$ under the 
identifications $(0,x)\sim s$ and $(t,s)\sim s$ for every $x\in \bd B$ and $t\in[0,1]$.
Then we can construct $f'$ from $f$ by defining it on $B$ by $f'((t,x)) = f(x)$ for each $(t,x)\in B\sm\{s\}$.
Note $f'$ is continuous since it is not defined on $s$. 
\end{proof}

\begin{corollary}
Every isolated point singularity $s\in\mc S$ of a $3$-frame field $f$ that is not on $\bd M$ is redundant.
\end{corollary}

\begin{proof}
In this case $s\in \Int(B)$ for some $n$-ball neighborhood $B$ containing no other singularities.  
\end{proof}

The same cannot be said for $2$-frame fields (which stems from the fact that $\mc O_2\cong S^1$
and $\pi_1(S^1)\cong\mb Z$).
Some additional conditions are needed when $s$ happens to be on the boundary. 
Let $g$ be the $2$-frame field on $\bd M$ obtained from $f$ 
by forgetting the axis normal to the boundary of each $3$-frame on $\bd M$. 

\begin{proposition}
If $f$ is boundary-aligned and $s$ is an isolated singularity on the boundary of $M$,
then $s$ is redundant if and only if $s$ is redundant as a singularity of the $2$-frame field $g$. 
\end{proposition}

\begin{proof}
Since $f$ is boundary aligned, $s$ is redundant for $g$ whenever it is for $f$ directly from definition.
Conversely, if it is redundant for $g$, then we can form a boundary aligned $3$-frame field $f'$ on 
$(N\sm \Int_M(B))\cup (B\cap\bd M)$ that agrees with $f$ on $(N\sm \Int(B))$. 
We are left to interpolate $f'$ from $A=(B\cap\bd M)\cup \bd B$ onto $B\sm A$. 
Notice $A\cong S^2$ and $B\sm A\cong \Int(D^3)$. 
Then such an interpolation exists using a nullhomotopy as in the proof of Proposition~\ref{PRedundant}.    
\end{proof}

Proposition~\ref{PRedundant} applies to reprove a result in~\cite{Li2012}:

\begin{corollary}
Every branch $L$ of a $3$-frame field $f$ that does not intersect $\bd M$ is redundant.
\end{corollary}

\begin{proof}
We can take a small enough tubular neighbourhood $\mc T$ of $L$ such that the closure of 
$\mc T$ is an $n$-ball $B$ containing $L$, and with only one singularity on the boundary of $B$, 
namely the at most one non-leaf vertex removed to obtain $L$. Then Proposition~\ref{PRedundant} applies.
\end{proof}

\begin{proposition}
If $\bd M$ contains only vertices of a singularity graph $\mc S$, then $f$ can be modified near $\bd M$ 
so that these vertices are incident to at most one edge.
\end{proposition}

\begin{proof}
Let $M'$ be $M$ with a a collar $C=\bd M\times[0,1]$ attached to $\bd M$ along $\{0\}\times\bd M$, 
and extend $f$ to $M'$ by defining $f(x,t)=f(x)$ for each $(x,t)\in\bd M\times\{t\}\subset C$. 
Since $M'\cong M$, we are done. 
\end{proof}

\section{Existence on Surfaces}

Since a boundary-aligned $3$-frame field on $M$ induces a $2$-frame field (cross field) on the surface $\bd M$
(simply by forgetting the axis of each frame that is normal to $\bd M$ where it is located), we begin with the 
question of existence of $2$-frame fields on general oriented surfaces.

\begin{proposition}
\label{PEuler}
If $\mc N$ is a compact connected surface and $\chi(\mc N)=0$, 
then there exists a boundary-aligned $2$-frame field on $\mc N$ without singularities.
\end{proposition}

\begin{proof}
A manifold $\mc N$ has a non-zero vector field if and only if its Euler characteristic $\chi(\mc N)$ is zero~\cite{Milnor, Hopf},
and a vector field on an oriented surface $\mc N$ induces a $2$-frame field simply by taking cross products with surface normals.
\end{proof}

A converse of this can be shown using the homology of $\widetilde{\mc O}_3$: 

\begin{proposition}
\label{PEuler2}
If $\mc N$ is a compact connected oriented surface,
then there exists a boundary-aligned $2$-frame field $g$ on $\mc N$ without singularities if and only if $\chi(\mc N)=0$
(namely, $\mc N$ must either be a torus $S^1\times S^1$ or an annulus $S^1\times [0,1]$).
\end{proposition}

\begin{proof}
Pick a smooth embedding of $\mc N$ in $\mb R^3$. The $2$-frame field $g$ induces a map
$$
\tilde g\wcolon\seqm{\mc N}{}{\widetilde{\mc O}_3}
$$
given for each $x\in \mc N$ by assigning the element of $\widetilde{\mc O}_3$ whose first distinguished axis goes
through the surface normal at $x$ given by the orientation on $\mc N$, and the other two axes the tangent cross given by $g(x)$.

Let us first suppose $\mc N$ is closed. Notice the composite 
$$
\tilde\rho\circ \tilde g\wcolon\seqm{\mc N}{}{S^2}
$$ 
with the quotient \seqm{\widetilde{\mc O}_3}{\tilde\rho}{S^2} the Gauss map of $\mc N$, 
whose degree is well known to be equal to $\frac{1}{2}\chi(\mc N)$. 
Thus, the map $(\tilde\rho\circ \tilde g)_*\colon\seqm{\mb Z\cong H_2(\mc N)}{}{H_2(S^2)\cong\mb Z}$ 
it induces on integral homology is multiplication by $\frac{1}{2}\chi(\mc N)$.
But $(\rho')_*=0$ on $H_2()$ since $H_2(\widetilde{\mc O}_3)=0$. 
Therefore $\chi(\mc N)=0$.

Now suppose $\mc N$ has non-empty boundary. Obtain a closed oriented surface $\mc N^+$  by gluing two copies of $\mc N$ 
along their common boundary via the degree $-1$ map \seqm{S^1}{-1}{S^1} on each boundary component. 
Since $\chi(S^1)=0$, the inclusion-exclusion principle implies $\chi(\mc N^+)=2\chi(\mc N)$. 
Since $g$ is boundary-aligned on both copies of $\mc N$, choosing a smooth embedding for $\mc N^+$, 
defines a frame field on $\mc N^+$ without singularities. Thus $\chi(\mc N^+)=0$, and we are done.
\end{proof}

\subsection{The Poincar\'e-Hopf Theorem}

A version of the Poincar\'e-Hopf theorem for boundary-aligned $2$-frame fields over oriented surfaces has been given 
in~\cite{Ray2008, Fogg2017, Beaufort2017}. This is stated in terms of a quarter-integer \emph{index} defined on point 
singularities of $2$-frame fields, the sum of which is related the Euler characteristic of the surface. 
Algorithms for constructing such frame fields with prescribed singularity constraints are also given in~\cite{Ray2006}.
We will give a somewhat different proof of the Poincar\'e-Hopf theorem for $2$-frame fields. We also prove the converse.
Namely, that a boundary-aligned $2$-frame field with a given pattern of singularities indexes exists when they sum to the 
Euler characteristic.

Fix $\mc N$ to be a compact oriented surface (possibly with boundary $\bd\mc N$), 
and $g$ a (not necessarily boundary-aligned) $2$-frame field with $\mc S$ consisting of finitely many isolated point 
singularities, all of which lie in the interior of $\mc N$.

Pick a sufficiently small $2$-disk neighbourhood $D_s\subseteq\mc N$ of each singularity $s\in\mc S$ so that $D_s$ 
contains no other singularities. Take a local trivialization \seqm{TD_s}{\cong}{D_s\times \mb R^2} of the tangent sub-bundle 
$TD_s \subseteq T\mc N$ such that it restricts to orientation preserving linear isomorphism \seqm{T\{x\}}{\cong}{\{x\}\times R^2}
on tangent planes $T\{x\}\cong \mb R^2$ for each $x\in D_s$ (the orientation on $\{x\}\times\mb R^2$ is taken to be clockwise).
This gives a trivialization $V_2(D_s)/\mc B_2 \cong D_s\times \mc O_2$ as a sub-bundle of $V_2(\mc N)$. Take the map 
$$
\kappa_s\wcolon\seqm{S^1\,\cong\, \bd D_s}{}{\mc O_2\,\cong\, S^1}
$$
defined as the composite 
$$
\seqmmm{\bd D_s}{g}{V_2(D_s)/\mc B_2}{\cong}{D_s\times \mc O_2}{}{\mc O_2},
$$
where the first map is the restriction of $g$ to $\bd D_s$, and the last map is the projection onto the second factor.
Then define
$$
ind_g(s) \,:=\, \frac{1}{4}deg(\kappa_s)
$$
where $deg(\kappa_s)$ is the homological degree of 
\seqm{\mb Z\cong H_1(\bd D_s)}{(\kappa_s)_*}{H_1(\mc O_2)\cong\mb Z}.
(\begin{remark} 
This can equivalently be taken to be the Brouwer degree if $g$ is a smooth map. 
Otherwise $g$ can be homotoped to a smooth map by Whitney approximation and the Brouwer degree used.
The index can also be thought of as integrating the signed angle change of a cross as it rotates about itself going clockwise around 
$\bd D_s$ with respect to the given orientation. 
\end{remark})

A notion of \emph{index} can be defined in a similar manner on each connected boundary component of $\mc N$ 
whenever $\bd\mc N$ is non-empty. 
Let $B\subseteq \bd\mc N$ be a connected boundary component and $C_B$ a closed collar neighbourhood of $B\cong S^1$. 
Since $C_B$ is an annulus $S^1\times [0,1]$, the tangent sub-bundle $TC_B$ is trivial. Then similarly as before we can define 
$$
\kappa_B\wcolon\seqm{S^1\,\cong\, B}{}{\mc O_2\,\cong\, S^1}
$$
and
$$
ind_g(B) \,:=\, \frac{1}{4}deg(\kappa_B).
$$
by projecting frames in $g$ that are on $B$ to the second factor of $V_2(\mc N)/\mc B_2\cong C_B\times \mc O_2$.
This trivialization of $V_2(\mc N)/\mc B_2$ depends on the choice of trivialization $TC_B\cong C_B\times \mb R^2$,
which can affect the value of $ind_g(B)$. 
We choose the one inherited from the standard trivialization of $T\mb R^2$ by embedding the annulus $C_B$ in $\mb R^2$ 
so that it bounds circles or radius $1$ and $2$, with $B$ embedded onto the inner circle of radius $1$. In this case we have
$$
ind_g(B) = ind_g(s)
$$
whenever $B = \bd D_s$ for some singularity $s$. On the other hand, if crosses on $B$ have axes normal to $B$,
then they rotate $4$ times back to themselves going around $B$ with respect to tangent planes in this trivialization.
So if $g$ is boundary-aligned on $B$, then
$$
ind_g(B) = 1
$$
with respect to this trivialization.

We make use of the well-known classification of compact surfaces: 
any closed connected oriented surface is homeomorphic to either a $2$-sphere $S^2$ or a connected sum $T\#\cdots\# T$ 
of torii $T=S^1\times S^1$; any oriented surface with boundary is homeomorphic to a closed oriented surface with \emph{holes},
each hole obtained by removing the interior of an arbitrarily small $2$-disk neighbourhood $D_s$ of a point $s$. 

\begin{proposition}
\label{PDeg}
Suppose $\mc N$ is a compact connected oriented surface. Let $B_0,\ldots,B_{\ell-1}$ denote the connected boundary components 
of $\bd\mc N$.
\begin{itemize}
\item[(i)] If there exists a (not necessarily boundary-aligned) $2$-frame field $g$ on $\mc N$ without singularities, then
$$
\chi(\mc N) + \ell\, =\, \csum{i\geq 0}{}{ind_g(B_i)}.
$$
\item[(ii)] Conversely, if we pick any integers $k_0,\ldots,k_{\ell-1}$ such that 
$$
\chi(\mc N)  + \ell\, =\, \csum{i\geq 0}{}{\frac{k_i}{4}},
$$
then there exists a (not necessarily boundary-aligned) $2$-frame field $g$ on $\mc N$ without singularities 
such that $ind_g(B_i) = \frac{k_i}{4}$ for each $i$.
\end{itemize}
\end{proposition} 

\begin{proof}[Proof of (i):]
The case where $\mc N$ is closed ($\ell=0$) follows from Proposition~\ref{PEuler2}, 
so let us suppose $\mc N$ has non-empty boundary. 
We proceed by induction on torus decomposition starting with $2$-spheres with holes. 

\textbf{Holed spheres:} Suppose $\mc N$ is a $2$-sphere with $\ell\geq 1$ holes, 
so $\bd\mc N$ consists of $B_i$'s forming the boundaries of these holes. 
Choose a smooth embedding $\xi$ of $\mc N$ into the standard unit $2$-disk $D^2$ in $\mb R^2$ 
such that the first component $B_0$ is the boundary of $D^2$, 
with the rest of the $B_i$'s forming boundaries of holes in the interior of $D^2$ (if any). 
Then our $2$-frame field is represented by a map \seqm{\mc N}{g}{\mc O^2\cong S^1}. 
Consider the degree $deg(g_{{\mathrel{|}}B_i})$ of the restriction of $g$ to the circle $B_i$. 
Notice $deg(g_{\mathrel{|}B_i}) = deg(\kappa_{B_i})$ when $i\geq 1$ and $deg(g_{\mathrel{|}B_0}) = deg(\kappa_{B_0})-8$.
This last equality is due to our embedding of the annulus $C_{B_0}$ defining $\kappa_{B_0}$ being inverted when we embed via $\xi$.
This causes a discrepancy between the standard bases of tangent planes in our chosen trivialization of $TC_{B_0}$ 
and those in the trivialization $TD^2$: 
upon embedding the standard bases of tangent planes on $B_0\subseteq C_{B_0}$ are rotated twice when going around $\bd D^2$, 
and each rotation corresponds to $4$ rotations of a cross to itself. 
Now recall the following well-known fact~\cite{Milnor}: 
if $X$ is a connected $n$-manifold, $M$ a compact oriented $(n+1)$-manifold with boundary, 
and a map \seqm{\bd M}{f}{X} extends to a map $\seqm{M}{}{X}$, then $deg(f)=0$. Therefore
$$
0\,=\, deg(g_{{\mathrel{|}}\bd\mc N}) \,=\,  \csum{i}{}{deg(g_{\mathrel{|}B_i})}
\,=\,  deg(\kappa_{B_0})-8+\csum{i\geq 1}{}{deg(\kappa_{B_i})},
$$
which implies $\chi(\mc N)+\ell = \chi(S^2) = 2 = \csum{i}{}{ind_g(B_i)}$.

\textbf{Gluing:} This implies a useful fact, which we use below. 
If $\mc N$ is obtained by gluing $Y$ and $Z$ along some common hole boundaries $B'\cong S^1$, 
then
\begin{equation}
\label{ECommon}
ind_f(B') + ind_h(B') = 2  
\end{equation}
where $f$ and $h$ are the restrictions of $g$ to $Y$ and $Z$. To see this, take the $2$-sphere with $2$ holes
$\Sigma_2\cong B'\times [0,1]$, and define a $2$-frame field $\delta$ on $\Sigma_2$ by $\delta(x,t)=g(x)$ for each $x\in B'$, $t\in [0,1]$.
Then using the above, $ind_\delta(B'\times\{0\}) + ind_\delta(B'\times\{1\}) = 2$, and assuming the correct orientation,
$ind_\delta(B'\times\{0\})=ind_f(B')$ and $ind_\delta(B'\times\{1\})=ind_h(B')$.

\textbf{Holed torii:} Next suppose $\mc N$ is a torus $T$ with $\ell\geq 1$ holes with boundaries $B_i$. 
We can form $\mc N$ by gluing a $2$-sphere with $2$ holes $\Sigma_2$ to a $2$-sphere with $\ell+2$ holes $\Sigma_{\ell+2}$ 
along circle boundaries $B'_1$ and $B'_2$ of two holes. Let $f$ and $h$ be the restrictions of $g$ to $\Sigma_2$ and $\Sigma_{k+2}$.
Then by the above we have $ind_f(B'_j)+ind_h(B'_j)=2$ for $j=1,2$ and 
$$
ind_f(B'_1) + ind_f(B'_2) \,=\, 2
$$
$$
ind_h(B'_1) + ind_h(B'_2) + \csum{i}{}{ind_g(B_i)}\,=\,2.
$$
Therefore $\chi(\mc N)+\ell = \chi(T) = 0 = \csum{i}{}{ind_g(B_i)}$.

\textbf{General case:} Finally, suppose $\mc N$ compact connected oriented surface with non-empty boundary.
If $\mc N$ is not a sphere or torus with holes, then $\mc N$ is homeomorphic to a connected sum
of torii $T_1\#\cdots\# T_k$ with $\ell$ holes poked in it, with some $\ell_j\geq 0$ holes of these lying in each torus $T_j$,
and the $B_i$ form the boundaries of these holes. 
Let $\mc N_1$ be $T_1\#\cdots\# T_{k-1}$ with $\ell-\ell_k+1$ holes: the holes $B_i$ that do not lie in $T_k$, 
together with an addition hole $B'\cong S^1$.
Similarly, let $\mc N_2$ be $T_{k-1}$ with $\ell_k+1$ holes: the holes $B_i$ lying in $T_k$, together with an additional hole $B'$.
Choosing $B'$ appropriately, $\mc N$ is formed by gluing $\mc N_1$ and $\mc N_2$ along $B'$, 
and our $2$-frame field $g$ restricts to frame fields $f$ and $h$ on these.
Inducting on number of torii in a decomposition, we may assume Proposition~\ref{PDeg} (i) holds for $\mc N_1$ and $\mc N_2$, so
$$
\chi(\mc N_1) + \ell-\ell_k+1\,\,\, =\,\,\, ind_f(B')+\csum{B_i\subseteq\bd\mc N_1}{}{ind_g(B_i)}
$$
$$
\chi(\mc N_2) + \ell_k+1\,\,\, =\,\,\, ind_h(B')+\csum{B_i\subseteq\bd\mc N_2}{}{ind_g(B_i)}.
$$
Then since $ind_f(B')+ind_h(B')=2$, 
by the inclusion-exclusion principle $\chi(\mc N)=\chi(\mc N_1)+ \chi(\mc N_2)$ since $\chi(B')=0$, 
and $\ell=\csum{i}{}{\ell_i}$, we obtain
$$
\chi(\mc N) + \ell\, =\, \csum{i\geq 0}{}{ind_g(B_i)}.
$$
\end{proof}

\begin{proof}[Proof of (ii):]
We proceed by induction as in part (i). 
Again, the $\ell=0$ case follows from Proposition~\ref{PEuler2}, so we assume $\mc N$ has non-empty boundary.

We will make use of the fact that $S^1$ is the \emph{Eilenberg-Maclane space} $K(\mb Z,1)$. 
This implies for any $CW$-complex $X$ there is a one-to-one correspondance 
$$
\zeta\wcolon\seqm{[X,S^1]}{}{H^1(X)}
$$
between the homotopy classes of maps \seqm{X}{f}{S^1} and integral cohomology groups $H^1(X)$.
Here $\zeta$ is given for any representative $f$ by $\zeta([f]) = f^*(\mu)$ where $\mu$ is a generator of $H^1(S^1)\cong\mb Z$ 
and \seqm{H^1(S^1)}{f^*}{H^1(X)} is the map induced by $f$ on cohomology.
We will find this useful for constructing $2$-frame fields for surfaces $X$ embedded in $\mb R^2$, 
since under the homeomorphism $S^1\cong \mc O_2$ the map \seqm{X}{f}{S^1} describes a $2$-frame field on $X$ in $\mb R^2$.

\textbf{Holed spheres:} Suppose $\mc N=\Sigma_\ell$ is a $2$-sphere with $\ell\geq 1$ holes embedded in $D^2\subseteq\mb R^2$ as in part (i).
Let \seqm{\bd B_i}{\iota_i}{\mc N} denote the inclusion of the hole boundary.
By the morphism of cohomology long exact sequences induced by the inclusion of pairs 
\seqm{(\mc N,\,\coprod_{i\geq 1}  \bd B_i)}{\rho}{(D^2,\,\coprod_{i\geq 1}  B_i)},
together with the fact that $\rho$ induces an isomorphism on cohomology by excision, we obtain
$$
H^1(\mc N)\cong\mb Z\{\beta_1,\ldots,\beta_{\ell-1}\}
$$ 
such that $(\iota_i)^*(\beta_i)=\nu_i$ and $(\iota_i)^*(\beta_j)=0$ if $i\neq j$, where $\nu_i$ is a generator of $H^*(B_i)\cong\mb Z$.
Consider the element $\omega =  k_1\beta_1+\cdots+k_{\ell-1}\beta_{\ell-1}$ in $H^1(\mc N)$ together with the homotopy class 
$\zeta^{-1}(\omega)\in [\mc N,S^1]$ (for $X=\mc N$), and let \seqm{\mc N}{g}{S^1} be a representative of this homotopy class. Then 
$$
(\iota_i)^*\circ g^*(\mu)\,=\,(\iota_i)^*(\omega)\,=\,k_i\nu_i.
$$
Thus, as $\Sigma_2$ is embedded in $\mb R^2$, $g$ is a $2$-frame field that restricts to a degree $k_i$ map on $B_i$ for $i\geq 1$ 
(i.e. $ind_g(B_i)=\frac{k_i}{4}$). The remaining constraint is for $B_0=\bd D^2$ to satisfy $ind_g(B_0)=\frac{k_0}{4}$.
But this follows now from part $(i)$ since: we must have $\chi(\mc N) + \ell\, =\, \csum{i}{}{ind_g(B_i)}$; 
we assume that $\chi(\mc N) + \ell\, =\, \csum{i}{}{k_i}$; and we found that $ind_g(B_i)=\frac{k_i}{4}$ for $i\geq 1$.

(\begin{remark}
The $2$-frame field $g$ can be constructed more concretely by using a homotopy equivalence \seqm{\mc N}{\aleph}{C} with the wedge sum of circles 
$C=\bigvee_{i\geq 1}B_i$. Here $\aleph$ can be taken so that each inclusion $\iota_i$ composes with $\aleph$ 
to the inclusion \seqm{B_i}{}{C} of the $i^{th}$ summand. Then a map \seqm{C}{q}{S^1} can easily be constructed so that
$g=q\circ \aleph$ is a frame field as above.   
\end{remark})

\textbf{Gluing:} The holed sphere case implies a useful fact, which we use below. 
Suppose $f$ and $h$ are $2$-frame fields on $Y$ and $Z$,
and $\mc N$ is obtained by gluing $Y$ and $Z$ along holes boundaries $B'$ such that $ind_f(B') + ind_h(B') = 2$.
Then there is a $2$-frame field $g$ on $\mc N$ satisfying $ind_g(B)=ind_f(B)$ or $ind_g(B)=ind_h(B)$,
for all hole boundaries $B$ of $\mc N$ that are holes boundaries of $Y$ or $Z$ respectively. To see this, 
note $\mc N$ can be identically obtained by gluing $Y$ and $Z$ to either end of a $2$-sphere with $2$ holes 
$\Sigma_2\cong B'\times [0,1]$ for each $B'$.
Using the above, there is a $2$-frame field $\delta$ on each $\Sigma_2$ such that 
$ind_\delta(B'\times\{0\})=ind_f(B')$ and $ind_\delta(B'\times\{1\})=ind_h(B')$. 
This implies $\delta$ is homotopic to $f$ and $h$ on either of the two boundary holes of $\Sigma_2$, 
so taking homotopies along collar neighbourhoods of these holes, 
we can modify $\delta$ so that it is equal to $f$ and $h$ on each hole. 
Then define $g$ by restricting to $f$, $h$, and $\delta$ on $Y$, $Z$, and $\Sigma_2$ respectively.

\textbf{Holed torii:} Returning to our induction, suppose $\mc N$ is a torus $T$ with $\ell\geq 1$ holes, as in part (i),
formed by gluing a $2$-spheres with holes $\Sigma_2$ and $\Sigma_{\ell+2}$ along hole boundaries $B'_1$ and $B'_2$.
Let $B_0,\ldots,B_{\ell-1}$ the holes of $\Sigma_{\ell+2}$ that are not either of the $B'_i$'s.
Using the above, our assumption $\csum{i}{}{\frac{k_i}{4}}=\chi(\mc N)+\ell$,
and the fact $\chi(\mc N)+\ell=\chi(T)=0$ and $\chi(\Sigma_j)+j=2$,  
we can construct $2$-frame fields $f$ and $h$ on $\Sigma_2$ and $\Sigma_{\ell+2}$ with
$ind_f(B'_i)=ind_h(B'_i)=\frac{k'_i}{4}$ and $ind_h(B_i)=\frac{k_i}{4}$ such that: 
$\frac{k'_1}{4}+\frac{k'_2}{4}=2$ and $\frac{k'_1}{4}+\frac{k'_2}{4}+\csum{i}{}{\frac{k_i}{4}}=2$.
Then as we showed above, $f$ and $h$ imply there exists a $2$-frame field $g$ on $\mc N$ with $ind_g(B_i)=\frac{k_i}{4}$.

\textbf{General case:} Finally, suppose $\mc N$ compact connected oriented surface with non-empty boundary,
obtained as in part (i) by gluing the connected sums of torii $\mc N_1$ and $\mc N_2$ along a hole boundary $B'$.
Pick $k'_1$ and $k'_2$ such that $\frac{k'_1}{4}+\frac{k'_2}{4}=2$. Since  $\chi(\mc N)=\chi(\mc N_1)+ \chi(\mc N_2)$, we have
$$
\chi(\mc N_1) + \ell-\ell_k+1\,\,\, =\,\,\, \frac{k'_1}{4}+\csum{B_i\subseteq\bd\mc N_1}{}{\frac{k_i}{4}}
$$
$$
\chi(\mc N_2) + \ell_k+1\,\,\, =\,\,\, \frac{k'_2}{4}+\csum{B_i\subseteq\bd\mc N_2}{}{\frac{k_i}{4}}.
$$
Inducting on number of torii in a connected sum decomposition we may assume Proposition~\ref{PDeg} (ii) holds for 
$\mc N_1$ and $\mc N_2$. So the above equalities imply there are $2$-frame fields $f$ and $h$ on $\mc N_1$ and $\mc N_2$
such that $ind_f(B_i)=\frac{k_i}{4}$ or $ind_h(B_i)=\frac{k_i}{4}$ if $B_i\subseteq\bd\mc N_1$ or $B_i\subseteq\bd\mc N_2$. 
So by the above we obtain our $2$-frame field $g$ such that $ind_g(B_i)=\frac{k_i}{4}$, and we are done.

\end{proof}

\begin{theorem}
\label{TPH}
Suppose $\mc N$ is a compact connected oriented surface of $\bd\mc N$.
\begin{itemize}
\item[(i)] If there exists a boundary-aligned $2$-frame field $g$ on $\mc N$ with finite singularities $\mc S$, then
$$
\chi(\mc N)\, =\, \csum{s\in \mc S}{}{ind_g(s)}.
$$
\item[(ii)] Conversely, if we pick any finite set of interior points $\mc S\subseteq \mc N$ and integers $k_s$ such that 
$$
\chi(\mc N) \, =\, \csum{s\in \mc S}{}{\frac{k_s}{4}},
$$
then there exists a boundary-aligned $2$-frame field $g$ on $\mc N$ with singularities $\mc S$
such that $ind_g(s) = \frac{k_s}{4}$ for each $s\in\mc S$.
\end{itemize}
\end{theorem}

\begin{proof}[Proof of (i)]
Let $k$ be the total number of connected boundary components of $\mc N'$.
For each $s\in\mc S$ take a small disk neighbourhood $D_s$ of $s$ such that no two of these nieghbourhoods intersect.
Then remove the interior of each $D_s$ from $\mc N$ to obtain a sub-surface $\mc N'$ having $\bd D_s$ as hole boundaries,
and with a total of $\ell=|\mc S|+k$ boundary components. 
Restricting the $2$-frame field $g$ to $\mc N'$, we have $ind_g(\bd D_s)=ind_g(s)$. 
For all other $k$ hole boundaries $B$ that are not any of $\bd D_s$, we have $ind_g(B)=1$ (since $g$ is boundary-aligned).
Since $\chi(\mc N) +k= \chi(\mc N') + \ell$, part(i) follows immediately from part (i) of Proposition~\ref{PDeg}. 
\end{proof}

\begin{proof}[Proof of (ii)]
Conversely, since $\chi(\mc N')+\ell=k+\chi(\mc N) =k+\csum{s\in \mc S}{}{\frac{k_s}{4}}$, 
by part (ii) of Proposition~\ref{PDeg} there is a (not necessarily boundary-aligned) $2$-frame field $g$ without singularities on $\mc N'$
satisfying $ind_g(\bd D_s)=k_s$ for each $s$, as well as $ind_g(B)=1$ for the other $k$ hole boundaries.
Then extend $g$ to $\mc N\sm\mc S$ by shrinking the frame field on $\bd D_s$ towards the center $s$ of $D_s$.   
\end{proof}

\begin{bigremark}
More is known for $1$-frame fields, which are popularly known as \emph{line fields}. With the appropriate notion of \emph{index} 
a Poincar\'e-Hopf theorem exists for line fields on compact smooth manifolds $M$. 
Moreover, a line field without singularities exists on $M$ if and only if there is also an everywhere non-zero vector field on $M$
(equivalently $\chi(M)=0$ if $M$ is connected)~\cite{Grant}.
\end{bigremark}

\section{Existence on $3$-manifolds}

Returning to $3$-frame fields, let $M$ be a $3$-manifold without corners smoothly embedded in $\mb R^3$,
and \seqm{M\sm\mc S}{f}{\mc O_3} a boundary-aligned $3$-frame field with singularity graph $\mc S$. 
We assume only the vertices of edges intersect $\bd M$, so $\mc S$ restricts to isolated point singularities on $\bd M$
at leaf vertices. Since $f$ can be modified locally so that the graph $\mc S$ has no interior branches, no isolated point 
singularities, and vertices on $\bd M$ are leaf vertices incident to exactly one edge, we can assume these properties hold 
for $f$ without losing too much generality. There are some additional assumptions on $\mc S$ made in the context of 
hex meshing in order to ensure a valid hex-only mesh:

\begin{enumerate}
\item[(1)] \textbf{Vertex alignment:} Frames very close to an interior vertex $v$ have an axis that point towards $v$. 
In other words, if $B_v$ is a sufficiently small ball neighbourhood of $v$ in $M\sm\bd M$, 
then frames on $\bd B_v\cong S^2$ have an axis that is normal to the surface $\bd B_v$ where they are located.
\item[(2)] \textbf{Edge alignment:} Frames very close to a the interior of an edge $e$ of the graph $\mc S$ 
have an axis that is tangent to $e$ where they are located. In other words, if $P$ is a plane through a point
$v$ in the interior of $e$ with normal a unit tangent vector $t_v$ of $e$ at $v$, then frames on a small $2$-disk 
neighbourhood $D\subset P$ of $v$ also have an axis normal to $P$. 
\end{enumerate}

In this case frames on $D\sm\{v\}$ induce a $2$-frame field $h_v$ on $D\sm\{v\}$ by forgetting the normal axis,
meaning we can define an \emph{edge index}
$$
ind_f(e) \,:=\, ind_{h_v}(v).
$$
Since a homotopy of $v$ along the interior of $e$ induces a homotopy of $h_v$,  
$ind_f(e)$ is independent of choice of $v$. Also If $ind_f(e)=0$, then the edge $e$ is redundant since the degree 
of $(h_v)_{\mathrel{|}\bd D}$ is $0$. This can be seen deleting the frame field in the interior of $D$, redefining it there 
using the nullhomotopy of $(h_v)_{\mathrel{|}\bd D}$, and then applying Proposition~\ref{PRedundant}. 
  
Let $\mc V$ denote the set of vertices of $\mc S$, $\mc V^\bd\subseteq \mc V$ the subset of vertices on $\bd M$, 
$\mc V^\circ=\mc V\sm\mc V^\bd$ the vertices in the interior, and $\mc E$ the set of edges of $\mc S$. 
For each vertex $v$, let $\mc E_v\subseteq \mc E$ denote the subset of edges that are incident to $v$. 

A central open problem when designing frame fields for hex-only meshing asks when a $3$-frame field can be extended to 
all of $M$ outside a chosen singularity graph $\mc S$ when there are prescribed contraints on its boundary and along singularities
as above~\cite{Liu2018, Crane2019, Palmer2020}. Again, they have be chosen carefully since they do not always lead to a valid 
hex-only mesh~\cite{Liu2018, Nieser2011}. Constructing the normal aligned frame field on the boundary surface is equivalent to 
constructing a surface cross field, which is a much easier problem solve. Some recent work has focused on extending the 
Poincar\'e-Hopf theorem for $2$-frame fields to $3$-frame fields~\cite{Liu2018, Nieser2011}. For example, the following
identities hold in terms of edge indices.

\begin{proposition}
If $f$ satisfies vertex and edge alignment, then the following identities must hold.
\begin{itemize}
\item[(i)] For each vertex $v$ in the interior of $M$:
$$
\csum{e\in \mc E_v}{}{ind_f(e)} \,\,=\,\, 2.
$$
\item[(i)] Over edges incident to every vertex:
$$
\csum{v\in \mc V}{}{\csum{e\in \mc E_v}{}{ind_f(e)}} \,\,\,=\,\,\, \chi(\bd M) + 2|\mc V^\circ| \,\,\,=\,\,\, 2\chi(M) + 2|\mc V^\circ|.
$$
\end{itemize}
\end{proposition}

\begin{proof}[Proof of (i):]
By vertex alignment $g$ induces a $2$-frame field $g_v$ on $\bd B_v\sm \mc A$ with point singularities $\mc A$
corresponding to the points on $\bd B_s\cong S^2$ where the incident edges in $\mc E_v$ intersect $\bd B_v$.
By edge alignment the index of these singularities is equal to the index $ind_f(e)$ of the edge $e\in \mc E_v$
passing through it. Then since $\chi(S^2)=2$, the result follows by Theorem~\ref{TPH}.
\end{proof}

\begin{proof}[Proof of (ii):]
Since $M$ is an odd dimensional manifold, $\chi(\bd M) = 2\chi(M)$ holds, and since $f$ is boundary-aligned, 
it restricts to an induced $2$-frame field $g$ on $\bd M$. 
For vertices on the boundary $v\in V^\bd$ there is only one edge $e$ in $\mc E_v$, and $ind_f(e)=ind_g(v)$, 
so it follows from Theorem~\ref{TPH} that
$$
\csum{v\in \mc V^\bd}{}{\csum{e\in \mc E_v}{}{ind_f(e)}} \,\,=\,\, \csum{v\in \mc V^\bd}{}{ind_g(v)} \,\,=\,\, \chi(\bd M).
$$
On the other hand, for vertices in the interior we have 
$$
\csum{v\in \mc V^\circ}{}{\csum{e\in \mc E_v}{}{ind_f(e)}} \,\,=\,\, \csum{v\in \mc V^\circ}{}{2} \,\,=\,\, 2|\mc V^\circ|.
$$
by part (i).
\end{proof}

This is only a necessary condition for $f$ to exist.  
In~\cite{Liu2018} a necessary and sufficient condition is given when boundary alignment and edge alignment constraints are present. 
This is stated in terms of a non-linear system of equations, against which an intricate \emph{chart zippering and merging} algorithm is 
provided to construct the desired field.
An alternate approach is taken in~\cite{Crane2019}, where a boundary and singularity constrained field is obtained from a continuous 
least squares optimization problem using a notion of derivative of moving frames across an edge (with a quadratic objective and linear 
constraints corresponding to tet edges on the boundary and along neighbourhoods of singularities).  
This has several advantages in practice, including the solution being independent of initialization and generation of new singularities being 
avoided away from the constraints. 

The condition we pursue is rooted in homotopy theory, this time involving a system of monomial equations consisting of discrete variables 
in the binary octahedral group $2\mc D_3$, each of which encode choices of elements in the fundamental group $\pi_1(\mc O_3)\cong 2\mc D_3$. 
Also, the system is given with respect to general $CW$-decompositions and a certain choice of spanning tree of its $1$-skeleton -- with variables 
corresponding to edges outside the tree. The use of $CW$-decompositions can simplify this system drastically compared to tetrahedral 
decompositions. At the same time, operations can be performed on any tetrahedral decomposition to turn it into a simpler $CW$-decomposition. 
This is at the expense of using information coming from the homotopy groups of $\mc O_3$, which is sometimes not completely explicit.

We apply this condition generally to $3$-frame fields with any given boundary and singularity constraints, not only those with 
boundary alignment, or edge and vertex alignment near $\mc S$. Namely, to any continuous choice of frames that have been fixed by the 
user on the boundary and near chosen singularities $\mc S$, which they wish to extend away from everywhere else into the 
interior without introducing any new singularities there. This is in a sense less flexible since a boundary frame field must be generated
first (say a cross field), then a singularity graph and constraints around it must be chosen that extend into the interior from the point 
singularities on the boundary. On the other hand, an aligned frame field on the boundary surface is easy to generate, and flexibly 
enough so that any given point singularity constraints that satisfy the Poincar\'e Hopf theorem for cross fields can be made hold. 
We also do not have to worry about alignment around corners, so we drop our assumption on $M$ being without corners from now on.

\subsection{A Necessary and Sufficient Condition}
\label{SNSC}
The concept of a $CW$-decomposition of a $n$-manifold $M$ generalizes tetrahedral or hex subdivisions.
This is defined in terms of a filtration 
$$
sk^0 M\subset sk^1 M\subset\cdots\subset sk^n M = M.
$$
The $i$-\emph{skeleton} $sk^i M$ is a subspace of $M$ of dimension $i$, built up from $sk^{i-1}M$ by gluing 
$i$-dimensional disks $D^i$ (called $i$-cells) to $sk^{i-1}M$ along their $(i-1)$-sphere boundaries via 
\emph{attaching maps} \seqm{S^{i-1}}{}{S^{i-1}}. For example, the $0$-skeleton $sk^0 M$ is a finite subset of 
points in $M$, and the $1$-skeleton $sk^1 M$ is a graph obtained from $sk^0 M$ by attaching edges to points at their 
ends. 

A $CW$-decomposition is \emph{regular} if each of its attaching maps are sphere homeomorphisms.
In this case, the graph $sk^1 M$ has edges attached to exactly two points (has no loops), the $2$-cells in $sk^2 M$ 
are glued to cycles in $sk^1 M$, and generally the boundaries of $i$-cells in $M$ are collections of $(i-1)$-cells 
patching together to form a topological $(i-1)$-sphere. A tetrahedral subdivision is an example of a regular 
$CW$-decomposition with facets glued along their boundary to exactly three edges forming a cycle in $sk^1 M$, 
and so on.

A \emph{spanning tree} $\mc T$ of a connected undirected graph $\mc G$ is a tree subgraph that contains 
all vertices of $\mc G$. A spanning tree can be constructed interatively by letting $\mc T^0$ be a single vertex, 
and $\mc T^i$ obtained from $\mc T^{i-1}$ by adjoining a choice of edge $e=\{v,w\}$ in $\mc G$ such that 
the vertex $v$ is in $\mc G -\mc T^{i-1}$, and the other vertex $w$ is in $\mc T^{i-1}$. 

\begin{lemma}
\label{LSpanning}
If $\mc H$ is a (not necessarily connected) subgraph of $\mc G$, then a spanning tree $\mc T$ can be constructed 
so that for each path-connected component $\mc H_j$ of $\mc H$, $\mc U_j:=\mc T\cap\mc H_j$ is a spanning tree 
of $\mc H_j$. 
\end{lemma}
\begin{proof}
Let $\mc U_j$ be any choice of spanning tree of $\mc H_j$.
We simply repeat the above construction of $\mc T$, but this time whenever $e=\{v,w\}$ is adjoined 
to $\mc T^i$ such that $v$ is a vertex of some $\mc H_j$ and $w$ in $\mc T^{i-1}$, then we also adjoin 
$\mc U_j$ to $\mc T^i$ by connecting them via $e$. This does not create any new cycles, so a tree is 
maintained throughout.
\end{proof}

Spanning trees are useful for describing fundamental group of connected graphs.
Namely, $\pi_1(\mc G)$ is the free group with generators in one-to-one correspondance to the edges in 
$\mc G-\mc T$. This is easy to see when we notice that the quotient map \seqm{\mc G}{}{\mc G/\mc T} collapsing 
$\mc T$ to point is a homotopy equivalence, since $\mc T$ is a contractible subcomplex of $\mc G$.
Then because $\mc G/\mc T$ is a graph with only one vertex, so it is a wedge of circles corresponding to the 
edges of $\mc G-\mc T$. 

Alternatively, we can see this more explicitly by picking a distinguised vertex $x$ from $\mc G$ as a basepoint, 
then for each edge $e=\{v,w\}$ in $\mc G-\mc T$, letting
\begin{equation}
\label{ECircuit}
\omega(e) := p_{x,v}\cdot \vec e \cdot p_{w,x}
\end{equation}
be the circuit formed by composing the directed edge $\vec e=(v,w)$ with simple paths $p_{x,v}$ and $p_{w,x}$ from 
$x$ to $v$ and $w$ to $x$. Since $\mc T$ is a spanning tree, these paths can be taken both lying wholly inside $\mc T$, 
and they are unique since $\mc T$ has no cycles. Thus $\omega(e)$ is well defined.
These circuits then generate $\pi_1(\mc G)$ since any circuit $c$ in $\mc G$ starting at $x$ is homotopic to the 
composition of circuits $\omega(e_0) \cdots \omega(e_k)$, where $e_0,\ldots,e_k,e_{k+1}=e_0$ are the edges of 
$\mc G-\mc T$ that $c$ goes through in order. The homotopy is given by contracting any shared chains of edges in 
$\omega(e_i)$ and $\omega(e_{i+1})$ that go from a shared vertex to $x$ and back again (note that edges in $c$ 
between $e_i$ and $e_{i+1}$ must be in $\omega(e_i)$ and $\omega(e_{i+1})$, otherwise we would have a cycle in 
$\mc T$).

\begin{lemma}
\label{LSpanning2}
For the tree $\mc T$ constructed in Lemma~\ref{LSpanning}, there is an ordering of the components $\mc H_j$
and a choice of vertex $b_j$ in each $\mc H_j$ such that for any vertex $x\in\mc H_j$:
(i) the simple path $p_{x,b_0}$ contains no vertices of $\mc H_k$ for $k>j$;
(ii) $p_{x,b_0}$ decomposes as $p_{x,b_0}=p_{x,b_j}\cdot p_{b_j,b_0}$;
(iii) the simple path $p_{x,b_j}$ is contained in the spanning graph $\mc U_j:=\mc T\cap\mc H_j$ of $\mc H_j$;
(iv) the simple path $p_{b_j,b_0}$ contains no edges and vertices in $\mc H_j$ except $b_j$.
\end{lemma}
\begin{proof}
Referring to the proof of Lemma~\ref{LSpanning}, we order the components so that $\mc H_j$ is the $j^{th}$ 
component adjoined and $b_j$ is the vertex $v$ in $\mc H_j$ that is connected to $\mc T^{i-1}$ by an edge $e$. 
Since the result is a tree and we haven't adjoined any $\mc H_k$ for $k>j$ at this point, $p_{x,b_0}$ exists and 
contains no vertices in $\mc H_k$.
The simple path $p_{b_j,b_0}$ starts by going through $e$ and into $\mc T^{i-1}$, and cannot go through $\mc H_j$ 
again since there are no other edges connecting it to $\mc T^{i-1}$. This also means $p_{x,b_0}$ must go through $b_j$,
and so decompose as $p_{x,b_j}\cdot p_{b_j,b_0}$. Lastly, $\mc U_j$ being a spanning tree of $\mc H_j$ and a 
subtree of $\mc T$, there is a simple path from $x$ to $b_j$ contained in both, which must be $p_{x,b_j}$ by 
uniqueness of simple paths.
\end{proof}

Next, we make precise the idea of a frame field being defined arbitrarily close to a singularity graph. 
Define a $3$-dimensional thickening of a graph $G$ as follows.

\begin{definition}
A $3$-\emph{thickening} $T_{\mc G}$ of graph $\mc G$ is a $3$-manifold built up as follows.
\begin{enumerate}
\item For each vertex $v$ in $G$ that is not a leaf, we adjoin a $3$-ball $D^3_v$.
\item For each leaf vertex $v$ in $G$, we adjoin a $2$-disk $D^2_v$.
\item For each edge $e=\{v_0,v_1\}$  in $G$ we adjoin a solid cylinder $C_e= D^2\times [0,1]$ by
gluing the ends $D^2\times \{i\}$, $i=1,2$, by either embedding into the $2$-sphere boundary of 
$D^3_{v_i}$ when $v_i$ is not a leaf, or embedding homeomorphically onto $D^2_{v_i}$ when it is a leaf. 
When not a leaf, we do this in such a way so that these embeddings do not overlap with any other 
cylinder ends embedded which correspond to any other edges incident to $v_i$.
\end{enumerate}
\end{definition}

Fix $M$ is a connected $3$-manifold (possibly with corners) embedded in $\mb R^3$ and $\mc S$ is a singularity graph 
embedded in $M$. As before, we can assume without much loss generality that $\mc S$ has no interior branches, 
no isolated point singularities, and singularities on the boundary $\bd M$ are point singularities corresponding to the leaf 
vertices of $\mc S$. Thus, we assume an embedding of a thickening $T_{\mc S}$ of $\mc S$ has been taking such that 
$\mc S$ is in the interior of $T_{\mc S}$, and:
\begin{enumerate}
\item vertices in the interior of $M$ are the non-leaf vertices and vertices on the boundary are the leaf vertices;
\item a non-leaf vertex $v$ is in the interior of $D^3_v$ and $D^3_v$ does not intersect the boundary $\bd M$; 
\item a leaf vertex $w$ is in the interior of $D^2_w$ and $D^2_v$ lies entirely inside $\bd M$;
\item each edge $e$ crosses the interior of the corresponding cylinder $C_e$, and $C_e$ does not intersect $\bd M$
anywhere except at either end when there is a leaf vertex there. 
\end{enumerate}

Let $M_{\mc S}$ be the submanifold of $M$ given by deleting from $M$ the interior of $T_{\mc S}$ as subspaces of $M$,
$$
M_{\mc S} := M \sm \Int_M(T_{\mc S}).
$$
Here we took the interior $\Int_M()$ as a subspace of $M$, not $\mb R^3$. Thus $\Int_M(T_{\mc S})$
contains the interior of the boundary intersection $T_{\mc S}\cap \bd M$ as subspaces of $\bd M$.
In this case, $M_{\mc S}$ is a $3$-manifold with corners along the surface boundary of $T_{\mc S}\cap \bd M$,
which is homeomorphic to a collection of $2$-disks, the corners being their circle boundaries. 
(The embedding of $\mc T_{\mc S}$ is analogous to the \emph{singularity tubes} about a singularity graph 
in~\cite{Liu2018, Crane2019}).

When we say a frame field is defined \emph{near the singularity graph} $\mc S$, we take this to mean that 
our embedding of $T_{\mc S}$ satisfying (1)-(4) above has been taken to be inside a small enough neighbourhood of
$\mc S$ so that the frame field is defined over all of $T_{\mc S}\sm\mc S$. For example, there is a frame field 
defined at least near the edges of $\mc S$ when frames satisfy edge alignment and edges are assigned an edge index.

Suppose there exists a $3$-frame field $f$ defined on $\bd M\sm\mc S$ and near the singularity graph $\mc S$.
Namely a map
$$
f\colon\seqm{(\bd M\sm\mc S)\cup (T_{\mc S}\sm\mc S)}{}{\mc O_3}
$$ 
defined on $T_{\mc S}\sm\mc S$, and on $\bd M$ except at point singularities corresponding to leaf vertices 
of $\mc S$. Notice $\bd M_{\mc S}$ is a submanifold of $(\bd M\sm\mc S)\cup (T_{\mc S}\sm\mc S)$, 
so $f$ retricts to a frame field
$$
f\colon\seqm{\bd M_{\mc S}}{}{\mc O_3}.
$$ 
Consider the following setup. Take:
\begin{enumerate}
\item Any regular $CW$-decomposition of $M_{\mc S}$  
$$
 sk^0 M_{\mc S}\subset sk^1 M_{\mc S}\subset sk^2 M_{\mc S}\subset sk^3 M_{\mc S} = M_{\mc S}
$$
(for example, any tetrahedral subdivision will do).
\item A spanning tree $\mc T$ of the graph $sk^1 M_{\mc S}$. Moreover, write $\bd M_{\mc S}$ as a disjoint union
$$
\bd M_{\mc S} := N_0 \sqcup\cdots\sqcup N_m
$$
of its path-connected components $N_i$. Then using Lemma~\ref{LSpanning}, and Lemma~\ref{LSpanning2}, 
assume $\mc T$ is taken such that 
$$
\mc U_i := \mc T\cap sk^1 N_i
$$ 
is a spanning tree of the subgraph $sk^1 N_i\subseteq sk^1 M_{\mc S}$, and the connected graph components 
$\mc H_i:=sk^1 N_i$ are ordered so that conditions (i) to (iv) in Lemma~\ref{LSpanning2} hold.
\item Denote the following differences of graphs
$$
\mc A \,:=\, sk^1 M_{\mc S}-\mc T
$$
$$
\bd\mc A \,:=\, sk^1 \bd M_{\mc S}-\mc T
$$
$$
\mc V_i \,:=\, sk^1 N_i-\mc U_i.
$$
Notice 
$$
\bd\mc A = \mc V_0 \sqcup\cdots\sqcup \mc V_m \subset\mc A.
$$
\item For each edge $e=\{v,w\}$ in $sk^1 M_{\mc S}$, choose a direction 
$$
\vec e=(v,w).
$$
\item Pick a basepoint vertex $b_i\in sk^1 N_i$ for each $i$, and ($\mc O_3$ being path-connected) let $\rho_i$ 
be any path in $\mc O_3$ from $f(b_i)$ to $f(b_0)$, with $\rho_0$ the constant path from $f(b_0)$ to itself.
\item When $e$ is an edge in $\mc V_i$, take the graph circuit in $sk^1 N_i$
$$
\phi(e) := p_{b_i,v},\cdot \vec e\cdot p_{w,b_i}
$$ 
where $p_{b_i,v}$ and $ p_{w,b_i}$ are the unique simple paths in the spanning tree $\mc U_i$ from $b_i$ to $v$
and from $w$ to $b_i$, and use this to define the loop in $\mc O_3$ based at $f(b_0)$
$$
\zeta(e) := \rho_i^{-1}\cdot f(\phi(e))\cdot \rho_i
$$
where $\rho^{-1}_i $ is the reverse path of $\rho_i$ from $f(b_0)$ to $f(b_i)$.
\item Denote $[\zeta(e)]\in 2\mc D_3$ the element of the binary octaheral group corresponding to the homotopy class of 
$\zeta(e)$ in the fundamental group $\pi_1(\mc O_3)\cong 2\mc D_3$ at basepoint $f(b_0)$.
\item Take
$$
\mc F := \{F_1,\ldots, F_\ell\}
$$ 
be the set $2$-dimensional cells of $sk^2 M_{\mc S}$ that do not lie enitrely inside the boundary $\bd M_{\mc S}$, 
i.e. are not $2$-cells in any $sk^2 N_i$ (in the case of a tetrahedral subdivision, the $F_i$ are non-boundary triangle facets).
\item For each $F\in\mc F$ pick a cycle of edges going either clockwise or counter-clockwise around $F$'s boundary
$$
\sigma_F:=(e_0,\ldots,e_{n_F})
$$
with vertex sequence 
$\bd\sigma_F:=(v_0,\ldots,v_{n_F},v_{n_F+1}=v_0)$; $e_j=\{v_j,v_{j+1}\}$.
\item For each edge $e$ on the boundary of $F$, let
$$
\delta^F_e: = 
\begin{cases}
1 & \mbox{if }\vec e\mbox{ points in the same direction as }\sigma_F;\\
-1 & \mbox{if }\vec e\mbox{ points in the opposite direction of }\sigma_F.
\end{cases}
$$
\end{enumerate}

With this we can phrase a necessary and sufficient condition for a frame field to extend from its boundary and from 
singularity constraints. We aim the proof to be as constructive as possible so that an algorithm can be extracted 
from it. The elements $[\zeta(e)]$ are used to encode frame field boundary and singularity constraints 
in $M$ in terms of the fundamental group $\pi_1(\mc O_3)\cong 2\mc D_3$ at basepoint $f(b_0)$, which reduce to only 
boundary constraints on the submanifold $M_{\mc S}$. The paths $\rho_i$ are there to fix change-of-basepoint 
isomorphisms between the different boundary components.

\begin{theorem}
\label{TMain}
A $3$-frame field $f$ defined on $\bd M\sm\mc S$ and near the singularity graph $\mc S$ extends to a frame field 
$f\colon\seqm{M\sm\mc S}{}{\mc O_3}$ defined over all of $M$ outside $\mc S$ if and only if 
there exists a $y_e\in 2\mc D_3$ for each edge $e$ in $\mc A-\bd\mc A$, 
and an $x_i\in 2\mc D_3$ for each boundary component $i=1,\ldots,m$, with $x_0:=1$ for the zeroth component,
such that the monomial equations have a solution in $2\mc D_3$ for each $F\in \mc F$
\begin{equation}
\label{EMonomial}
\cprod{e\,=\,e_1,\ldots,e_{n_F}\,\in\,\sigma_F}{}{z_e^{\delta^F_e}}\,=\, 1
\end{equation}
where
$$
z_e  := 
\begin{cases}
y_e & \mbox{if }e\,\in\,$\mc A-\bd\mc A$;\\
x_i^{-1}[\zeta(e)]x_i & \mbox{if }e\,\in\,\mc V_i\subseteq\bd\mc A;\\
1 & \mbox{otherwise}.
\end{cases}
$$
Note the product here is not commutative since $2\mc D_3$ is not abelian, so the order of the factors in the 
monomials matter.
\end{theorem}

\begin{proof}[Proof part(i):]
We first prove this condition is necessary. Suppose $f\colon\seqm{M\sm\mc S}{}{\mc O_3}$ exists.
Notice $f$ restricts to a frame field on the submanifold $M_{\mc S}\subset M\sm\mc S$.

Take any $F\in \mc F$ and denote $\sigma:=\sigma_F=(e_0,\ldots,e_{n_F})$.
Note $\sigma$ is a nullhomotopic loop in $F$ since we can homotope it to a point by contracting to the interior of $F$. 
Since the frame field $f$ is defined over the interior of each $F$, $f$ maps this nullhomotopy to a nullhomotopy of the 
loop $f(\sigma)$ in $\mc O_3$. In other words $[f(\sigma)]=0$ in $\pi_1(\mc O_3)$ (at any basepoint on $\bd F$).

For edges $e=\{v,w\}$ in $\mc A$, take the circuit
$$
\omega(e) := p_{b_0,v}\cdot \vec e \cdot p_{w,b_0}
$$
when $\vec e=(v,w)$, where $p_{v,w}$ denotes the unique simple path in the spanning tree $\mc T$ from vertex $v$ to $w$. 
A reverse path $p_{b,a}$ of $p_{a,b}$ is sometimes denoted $p^{-1}_{a,b}$. 

Suppose two edges $e$ and $e'$ in the cycle $\sigma$ are in $\mc A$, and every edge in a simple path $p_{v,w'}$ in $\sigma$ 
between these (going either way around $\sigma$) is not in $\mc A$, in other words $p_{v,w'}$ is in the spanning tree $\mc T$. 
We have three simple paths $\alpha_1:=p_{v,w'}$, $\alpha_2:=p_{v,b_0}$, and $\alpha_3:=p_{w',b_0}$ in $\mc T$ 
between three vertices $v$, $w'$, and $b_0$. The union of $\alpha_2$ and $\alpha_3$ is a subtree of $\mc T$ that contains a 
simple path from $v$ to $w'$, which must be $\alpha_1$ by uniqueness of simple paths in $\mc T$. There is then a distinguished 
vertex $u$, known as the \emph{median} of $v$, $w'$, and $b_0$, such that our paths decompose as
$$
\alpha_1 = \beta_1\cdot \beta_2
$$
$$
\alpha_2= \beta_1\cdot \beta_3
$$
$$
\alpha_3 = \beta_2^{-1}\cdot \beta_3
$$
in terms of simple paths $\beta_1:=p_{v,u}$, $\beta_2:=p_{u,w'}$, and $\beta_3:=p_{u,b_0}$ having only the vertex $u$ in 
common. We then have a homotopy of paths from $v$ to $w'$
\begin{equation}
\label{EMedian}
\alpha_1 \,=\, \beta_1\cdot \beta_2 \,\simeq\, \beta_1\cdot \beta_3 \cdot \beta_3^{-1} \cdot \beta_2 \,=\, \alpha_2\cdot \alpha_3^{-1}
\end{equation}
where the endpoints are $v$ and $w'$ are fixed in the homotopy, and the homotopy is within $\mc T$.

Let $0 \leq k_0 < \cdots < k_r \leq n_F$ be the subsequence indexing the edges $e_{k_j}$ in $\sigma$ that are edges in 
$\mc A$, and let $\alpha_{1,j}$ be the simple path in $\sigma$ between $e_{k_{j-1}}$  and $e_{k_j}$ (where $k_{-1}:=k_r$), 
which decomposes as $\alpha_{2,j}\cdot \alpha_{3,j}^{-1}$ as in~\ref{EMedian}. 
Letting $\bigvee_i c_i$ denote a composition $c_1\cdot c_2\cdot c_3\cdots$ of paths $c_i$ end-to-end, 
and writing $\delta_i:=\delta^F_{e_i}$ and $\sigma$ as a composition of paths of single directed edges $\vec e_i$,
\begin{equation}
\label{EBouquet}
\sigma \,=\,  \cvee{i=0,\ldots,n_F}{}{\vec e_i^{\delta_i}}  
\,= \, \cvee{j=0,\ldots,r}{}{(\alpha_{1,j}\cdot\vec e_{k_j}^{\delta_{k_j}})}
\,\simeq\, \cvee{j=0,\ldots,r}{}{(\alpha_{3,j}^{-1}\cdot\vec e_{k_j}^{\delta_{k_j}}\cdot \alpha_{2,j+1}) }
\,=\, \cvee{j=0,\ldots,r}{}{\omega(e_{k_j})^{\delta_{k_j}}}.
\end{equation}
The last of these is a bouquet of loops $\omega(e_{k_j})^{\delta_{k_j}}$ based at $b_0$, which $f$ maps to the bouquet of loops 
$\bigvee_j f(\omega(e_{k_j}))^{\delta_{k_j}}$ in $\mc O_3$ based at $f(b_0)$, which is (unbased) homotopic to $f(\sigma)$. 
Since $f(\sigma)$ is nullhomotopic,
$$
1 \,=\, [f(\sigma)] \,=\, \cprod{j=0,\ldots,r}{}{[f(\omega(e_{k_j}))]^{\delta_{k_j}}}
$$
in $\pi_1(\mc O_3)\cong 2\mc D_3$. Also, notice when $e:=e_{k_j}$ is in $\bd\mc A$ (i.e. is on $\bd M_{\mc S}$ and not in $\mc T$), 
both its endpoint vertices are in $\mc U_i$ for some $i$, so by Lemma~\ref{LSpanning2} both $\alpha_{2,j+1}$ and $\alpha_{3,j}$ decompose 
as simple paths that first go within $\mc U_i$ from either vertex of $e$ to the vertex $b_i$, then both go from $b_i$ to $b_0$ via the 
same simple path $p_{b_i,b_0}$ ($b_i$ is the median vertex here). So we have $\omega(e)=p_{b_i,b_0}^{-1}\cdot \phi(e)\cdot p_{b_i,b_0}$ 
where $\phi(e)$ is a cycle in $sk^1 N_i$ through $e$ and $b_i$. Then taking
$$
\chi_i\,:=\, \rho_i^{-1}\cdot f(p_{b_i,b_0}).
$$
we have
$$
f(\omega(e)) \,\simeq\, \chi_i^{-1}\cdot \rho_i^{-1}\cdot f(\phi(e))\cdot \rho_i\cdot \chi_i \,=\, \chi_i^{-1}\cdot\zeta(e)\cdot \chi_i.
$$ 
So $[f(\omega(e))]=[\chi_i]^{-1}[\zeta(e)][\chi_e]\in\pi_1(\mc O_3)$ when $e$ is in $\bd\mc A$.
Letting $x_i=[\chi_i]$  and $y_{e}=[f(\omega(e))]$ when $e\in \mc A-\mc\bd A$ finishes the proof.
(notice $x_0=1$ since $f(p_{b_0,b_0})$, $\rho_0$, and $\chi_0$ are constant paths at $f(b_0)$).

\end{proof}

\begin{proof}[Proof part(ii):]
Now we show the condition is sufficient. Most of the setup and notation is recycled from part (i). 

Since $f$ is defined on $\bd M\sm\mc S$ and near the singularity graph $\mc S$, $f$ restricts to a frame field
$f\colon\seqm{\bd M_{\mc S}}{}{\mc O_3}$. Our first task is to extend this frame field to all of $M_{\mc S}$ 
when the monomial equations~\ref{EMonomial} have a solution for every $F\in \mc F$. For the solutions 
$x_j,y_e\in2\mc D_3\cong\pi_1(\mc O_3)$ to these equations, we let $\chi_j$ and $\gamma_e$ denote a
choice of loop based at $f(b_0)$ in $\mc O_3$ that represents the homotopy classes of $x_j$ and $y_e$, 
in other words 
$$
x_j = [\chi_j]
$$ 
$$
y_e = [\gamma_e]
$$ 
where $\chi_0$ can be taken as the constant path at $f(b_0)$ since $x_0=1$.
Since $f$ is already defined on all vertices, edges, and $2$-cells lying on the boundary $\bd M_{\mc S}$, 
we focus on those in the interior, working inductively up skeleta by extending from vertices, to edges, to $2$-cells,
and finally to $3$-cells. 

For each vertex $v$ in $sk^1 M_{\mc S}$ that is not in $\bd M_{\mc S}$, choose any $o_v\in\mc O_3$ and define 
$f(v):=o_v$. To extend this to edges, consider the set
$$
\mc P:=\{d_1:=\{b_1,u_1\},\ldots,d_m:=\{b_m,u_m\}\}
$$
of edges in $\mc T$ such that $d_j$ is the first edge in the simple path $p_{b_j,b_0}$.
By our assumptions on $\mc T$ coming from Lemma~\ref{LSpanning2}, we have $d_j\neq d_k$ when $j\neq k$
since $p_{b_j,b_0}$ contains no vertices of $sk^1 N_k$ when $k>j$. 
Define $f(\vec e)$ for every edge $e$ in $\mc T$ and not in $\bd M_{\mc S}$  (in order) as follows: 
\begin{itemize}
\item[(a)] First, for those $e\nin\mc P$, define $f(\vec e)$ to be any choice of path in $\mc O_3$ from 
$f(v)$ and $f(w)$ where $\vec e=(v,w)$ (which exists since $\mc O_3$ is path-connected). 
\item[(b)] Set $u_0:=b_0$, $\vec d_0$ the empty path from $b_0$ to itself.
Then for each $d_j\in\mc P$, $j=1,\ldots,m$, set $\vec d_j=(b_j,u_j)$, and for each $j$ define the path 
$$
f(\vec d_j)\,:=\,\rho_j\cdot\chi_j\cdot f(p_{u_j,b_0})^{-1}
$$ 
 from $f(b_j)$ to $f(u_j)$.
We define these going in order from $j=0,\ldots,m$ so that $f$ is defined on every edge of $p_{u_j,b_0}$ 
as we go since (by Lemma~\ref{LSpanning2})  $p_{b_j,b_0}=\vec d_j\cdot p_{u_j,b_0}$  contains none of 
the edges $d_k$ for $k>j$. (Note $f(\vec d_0)$ is just the constant path at $f(b_0)$).
\end{itemize}
Lastly, for those edges $e$ that are both not in $\mc T$ (are in $\mc A$) and not in $\bd M_{\mc S}$, define
\begin{equation}
\label{EInterior}
f(\vec e)\, :=\, f(p_{v,b_0})\cdot \gamma_e\cdot f(p_{w,b_0})^{-1}
\end{equation}
where $\vec e=(v,w)$, and the paths $f(p_{v,b_0})$ and $f(p_{w,b_0})$ are already defined from above since the simple 
paths $p_{v,b_0}$ and $p_{w,b_0}$ consist of edges in $\mc T$. Note $f(\vec e)$ restricts to $f(v)$ and $f(w)$ on its 
vertices as required. This defines $f$ on all vertices and edges.

We are left to extend $f$ to every $2$-cell $F$ not in $\bd M_{\mc S}$ (i.e. the $2$-cells in $\mc F$), then to every 
$3$-cell. Suppose any given $F\in\mc F$ has the monomial equation~\ref{EMonomial} satisfied. Take the homotopy of
$\sigma:=\sigma_F$ to the bouquet of loops $\omega(e_{k_j})^{\delta_{k_j}}$ in~\ref{EBouquet} with each $e_{k_j}$
on the boundary of $\sigma$ and not in $\mc T$; $\omega(e) := p_{v, b_0}^{-1}\cdot \vec e \cdot p_{w,b_0}$ when 
$\vec e=(v,w)$. Notice there is at least one such $e_{k_j}$, or else we would have a cycle in $\mc T$ going around
the boundary of $F$. Also, notice when $e:=e_{k_j}$ is in $\mc A-\bd\mc A$ by~\ref{EInterior} there is a homotopy
$$
f(\omega(e)) \,=\,  f(p_{v,b_0})^{-1}\cdot f(\vec e) \cdot f(p_{w,b_0})  \,\,\simeq \,\, \gamma_e.
$$
So $[f(\omega(e))] = [\gamma_e] = y_e$.
On the other hand, when $e$ is in $\mc V_i\subseteq \bd\mc A$ for some $i$, as we saw in part (i)
$\omega(e)=\beta^{-1}\cdot \phi(e)\cdot \beta $ where $\beta:= p_{b_i,b_0}$  and $\phi(e)$ is a cycle
in $sk^1 N_i$. From $(b)$ we have $\beta = \vec d_i\cdot p_{u_i,b_0}$  and
$f(\vec  d_i)=\rho_i\cdot\chi_i\cdot f(p_{u_j,b_0})^{-1}$,  so
$f(\beta) = f(\vec d_i)\cdot f(p_{u_i,b_0}) \simeq \rho_i\cdot \chi_i$, and
$$
f(\omega(e)) \,=\,  f(\beta)^{-1}\cdot f(\phi(e)) \cdot f(\beta)  
\,\,\simeq \,\, \chi_i^{-1}\cdot\rho_i^{-1}\cdot f(\phi(e))\cdot \rho_i\cdot\chi_i \,=\,  
\chi_i^{-1}\cdot \zeta(e)\cdot\chi_i.
$$
Thus $[f(\omega(e))] = x_i^{-1}[\zeta(e)]x_i$.  We have shown that $[f(\omega(e))] = z_e$ for those $e:=e_{k_j}$, 
$j=1,\ldots,r$ on the boundary of $F$ that are not in $\mc T$, and since $z_e$ is just $1$ for the rest of the edges,
$$
[f(\sigma)] \,=\, \cprod{j=0,\ldots,r}{}{[f(\omega(e_{k_j}))]^{\delta_{k_j}}} \,=\,
\cprod{e\,=\,e_1,\ldots,e_{n_F}\,\in\,\sigma}{}{z_e^{\delta^F_e}}\,=\, 1
$$
which implies $f(\sigma)$ is nullhomotopic. We can then extend $f$ from the edges on the boundary cycle $\sigma$ of the $2$-cell 
$F$ into the interior of $F$ by following this nullhomotopy. Namely, $F\cong D^2\cong (S^1\times [0,1])/(S^1\times \{1\})$, 
and $f$ maps  $S^1\times \{t\}$ via nullhomotopy at parameter $t$, mapping to a single point at $t=1$. 
This defines $f$ on every $2$-cell in $\mc F$, and so $f$ is now defined on every $2$-cell in $sk^2 M_{\mc S}$.

Finally, to extend $f$ to the interiors of $3$-cells of $sk^3 M_{\mc S} = M_{\mc S}$, since $\pi_ 2(\mc O_3)=0$, 
$f$ is nullhomoptic on the $2$-sphere boundaries of each $3$-cell $C\cong D^3\cong (S^2\times [0,1])/(S^2\times \{1\})$,
so we can extend $f$ into their interior similar as we did for the $2$-cells. 

This defines $f$ on all of $M_{\mc S}$. Since $M\sm\mc S\cong M_{\mc S}\cup (T_{\mc S}\sm \mc S)$ glued along $\bd  M_{\mc S}$,  
and $f$ is defined on all of $T_{\mc S}\sm \mc S$ such that it agrees with $f$ on $\bd  M_{\mc S}$ where it is glued, we obtain a 
frame field $f\colon\seqm{M\sm\mc S}{}{\mc O_3}$.

\end{proof}

\begin{bigremark}
\label{RTMain}
The statement for Theorem~\ref{TMain} remains true when any $\bd M$ is replaced with a compact submanifold $N$ embedded smoothly 
in $\bd M$. The proof is independent of which submanifold is taken, we only require our chosen $CW$-decomposition of $M_{\mc S}$ to 
restrict to a $CW$-decomposition of $N$. More generally, $N$ can be any subcomplex of $\bd M_{\mc S}$ in the $CW$-decomposition, 
not necessarily a submanifold. 
Also, the only intrinsic properties of $\mc O_3$ used in the proof are: $\pi_1(\mc O_3)\cong 2\mc D_3$ and $\pi_2(\mc O_3)=0$,
and that $\mc O_3$ is path connected. The space $\mc O_3$ could just as well be replaced with any other path connected space $X$ 
with trivial second homotopy group, as long as $2\mc D_3$ is at the same time replaced with $\pi_1(X)$.
\end{bigremark}

\subsection{Determining homotopy classes of boundary constraints}

Before the algebraic system in Theorem~\ref{TMain} can be solved, the elements $[\zeta(e)]$ in $2\mc D_3\cong\pi_1(\mc O_3)$ that 
encode boundary and singularity constraints must be determined. To this end we can use the monodromy for the $2\mc D_3$-fibered universal 
covering $\seqm{S^3}{r}{\mc O_3}$. Here $r:=p\circ \bar p$ factors through the covers $\seqm{S^3}{\bar p}{SO(3)}$ and
$\seqm{SO(3)}{p}{\mc O_3}$ with fibers $\mb Z_2$ and $\mc D_3$ respectively. Picking a frame $\aleph\in \mc O_3$ and a unit quaternion 
$\mf q\in S^3$ as basepoints such that $r(\mf q)=\aleph$, the elements in the $2\mc D_3$-fiber at $\mf q$ correspond to the $48$ unit quaternions 
obtained from $24$ matrices $\mf m_i$ in $SO(3)$, each of which is given by positive determinant column permutations and reflections of the matrix 
$\bar p(\mf q)$.
Each quaternion among the $48$ is then one of $\bar p^{-1}(\mf m_i)=\{\mf a_i, -\mf a_i\}$ for some $i$, with $\mf q$ corresponding to the the identity
 element in $2\mc D_3$.

Assuming we chose $\aleph$ so that the loop (closed path) $\zeta(e)$ starts at basepoint frame $\aleph:=f(b_0)$, by the path lifting property
of the covering $r$, $\zeta(e)$ lifts to a unique path $\gamma$ in $S^3$ that starts at $\mf q$. The end point of $\lambda$ is then another 
quaterion $\mf q'$ lying in the same fiber as $\mf q$, i.e. $r(\mf q') = \aleph$. If $\mf q$ corresponds to the identity in its $2\mc D_3$-fiber, 
then the homotopy class $[\zeta(e)]\in\pi_1(\mc O_3)$ is exactly $\mf q'$ as an element of the $2\mc D_3$-fiber. 

To obtain $\mf q'$ explicitly all we have to do is construct the lift $\gamma$. This can be done concretely by subdividing the path $\zeta(e)$ into 
sufficiently small paths $I_i$ joined end-to-end, then defining the lift $\lambda_i\subset S^3$ of $I_i$ inductively in order starting with $i=1$ at the 
basepoint. Supposing $\lambda_{i-1}$ is already defined so that $\lambda_i$ is defined on its first endpoint $a_i$, if $I_i$ is taken small enough 
we can form a path $\lambda_i$ in $S^3$ as the subset of unit quaternions in $r^{-1}(I_i)$ that are within a small $\epsilon$ distance of $a_i$. 
Here $\epsilon$ is fixed before subdividing to be within a good margin of safety smaller than the smallest distance between any two quaternions
that are in the same $2\mc D_3$-fiber. Each quaternion $a$ along $\lambda_i$ is then the nearest quaternion to $a_i$ in its $2\mc D_3$-fiber. 
The union of the $\lambda_i$ defines our path $\lambda$.

\subsection{Solving the system}

The next difficulty is in finding a $CW$-decomposition of $M_{\mc S}$ and a solution to the system monomial equations~\ref{EMonomial} 
with respect to this decomposition -- given a solution exists. While there are robust algorithms that can generate a tetrahedral decomposition  
of $M_{\mc S}$ for us, the resulting decomposition will likely have a large number of facets, so we end up with a complicated system to solve. 
A regular $CW$-decomposition can be drastically simpler in contrast. Take for instance a solid $3$-ball, which can be decomposed with 
$2$, $2$, $2$, and $1$ numbers of cells in dimensions $0$, $1$, $2$, and $3$, or a solid torus, which can be decomposed with $4$, $8$, $6$, 
and $2$ numbers cells in dimensions $0$, $1$, $2$, and $3$. In the case of the solid torus, there are only two $2$-cells that are not entirely 
on the boundary, so the system consists of only two monomial equations. In the case of the $3$-ball the system is empty since there are no 
$2$-cells that are not entirely on the boundary, so the system is satisfied vacuously, and a frame field can always be extended from one that 
is defined everywhere on its $2$-sphere boundary (though not boundary aligned, so there are no boundary point singularities). This also follows 
from $\pi_2(\mc O_3)=0$.

Fortunately we can modify any given tetrahedral decomposition into a simpler regular $CW$-decomposition. 
This can be done iteratively by first merging adjacent tets to turn them into generic cells, then merging newly-created adjacent cells, while the system 
is simplified in the process. Merging two adjacent tets by removing their shared facet does not change the decomposition being regular, and removes 
one equation corresponding to the facet they share, leaving a generic $3$-cell in their place. Iterating this process, we end up merging adjacent 
$3$-cells as long as the patch-work of $2$-cells that they share on their boundary is homeomorphic to a $2$-disk. The shared $2$-cells
are removed along with the edges and vertices they contain that are not on the boundary of this patch-work. The result of this removal is then 
still a $3$-cell, and the $CW$-decomposition remains regular. In the process every monomial equation corresponding to the shared $2$-cells is 
removed from the system. Once we have reached the limit of merging $3$-cells, adjacent $2$-cells can in some cases be merged in a similar manner, 
as long as they both lie on the boundary between two $3$-cells and they share $1$-cells on their boundaries that form a single continuous path. 
This time the monomials corresponding to these cells are merged into a single monomial. 
Finally, paths of edges can be merged into a single edge when each edge in the path lies on the boundary of the same set of $2$-cells and $3$-cells
-- though our monomials will remain the same in this case since at most one of the path edges is not in the spanning tree.
While the system becomes simpler as we go, some of the setup has to be recomputed in the end (e.g. the classes $[\zeta(e)]$). 
Note that aside from this merging process there are well developed mathematical tools for simplifying $CW$-decompositions, namely 
\emph{Discrete Morse theory}~\cite{forman}.

\subsection{Simplifying boundary constraints}

As much as the system is simpler when the $CW$-decomposition is efficient, one possible limiting factor is the presence of boundary 
constraining classes $[\zeta_e]$. In the not-so-interesting scenario when there no boundary constraints the system can be solved trivially by setting all variables 
to $1$, so in this direction, it may be helpful to get rid of any $[\zeta_e]$ that are in some sense redundant. To see that this is plausible, suppose instead of 
asking in Theorem~\ref{TMain} for a frame field $f\colon\seqm{M\sm\mc S}{}{\mc O_3}$ that extends one constrained on the boundary and near singularities,
we ask for one that extends it \emph{only up-to-homotopy}. Namely, we want $f$ to exist on all of $M_{\mc S}$ such that the restriction
$f\colon\seqm{\bd M_{\mc S}}{}{\mc O_3}$ to the boundary surface $\bd M_{\mc S}$ is only homotopic to a provided frame field $f'$ defined on 
$\bd M_{\mc S}$. This is almost as good as the original statement since we can attach a collar $\bd M_{\mc S}\times [0,1]$ to $\bd M_{\mc S}$
by gluing along $\bd M_{\mc S}\times \{0\}$ (the space resulting from the gluing is still homeomorphic $M_{\mc S}$ by the collar neighbourhood theorem), 
then we define a frame field adhering to the original boundary constraints $f'$ by mapping $\bd M_{\mc S}\times \{t\}$ to $\mc O_3$ along the homotopy 
$f_{|\bd M_{\mc S}}\cong f'$. With this loosening of constraints we might expect a smaller a number of boundary constraining classes in this reformulation of 
Theorem~\ref{TMain}. At the same time we can avoid having to predefine a surface frame field or cross field everywhere, only where our simplified constraints 
tell us to.

To this end we should start by looking at homotopy classes of maps \seqm{\mf S_g}{}{\mc O_3} where $\mf S_g$ is a closed orientable surface of genus $g$ 
(among which $\bd M_{\mc S}$ belongs). 
The related problem of counting homotopy classes of cross fields has been studied in~\cite{Ray2008}; if $\mf S_{g,b}$ is a surface of genus $g$ with $b$ holes 
bounded by circles, then the set of homotopy classes of cross fields on $\mf S_{g,b}$ is isomorphic to $\mb Z^{2g+b-1}$. Our problem for boundary frame fields is 
actually a bit simpler since there are no holes (there are no point singularities, singularity graphs simply add handles to $\bd M_{\mc S}$ upping its genus), 
and we aren't restricting ourselves to surface alignment anywhere -- at least not strictly since everything is up-to-homotopy. 

Let $\vbr{X,Y}$ denote the set of basepointed homotopy classes of maps \seqm{X}{}{Y} between spaces $X$ and $Y$. 
Recall a closed orientable surface $\mf S_g$ of genus $g$ is homeomorphic to the $CW$-complex obtained by attaching a $2$-cell to a bouquet of 
$2g$ circles $W_{2g}:=\bigvee^{2g} S^1$. The attaching map for the $2$-cell
$$
\alpha\wcolon\seqm{S^1}{}{W_{2g}}
$$
is in the homotopy class corresponding to the product $[a_1,b_1]\cdots[a_g,b_g]$ in the free group on $2g$ generators
$\pi_1(W_{2g})\cong \vbr{a_1,b_1,\ldots,a_g,b_g}$ in $\mb Z$, where $[a_i,b_i]$ denotes the commutator $a_i b_i a^{-1}_i b^{-1}_i$. 
Consider the following two properties of maps from $\mf S_g$ to $\mc O_3$:
\begin{enumerate}
\item Since $\mf S_g$ is the homotopy cofiber of $\alpha$, by the homotopy extension property of homotopy cofibrations a map \seqm{W_{2g}}{\bar f}{\mc O_3} 
extends to a map \seqm{\mf S_g}{f}{\mc O_3} if and only if $\bar f\circ \alpha$ is nullhomotopic. [Proof:, take the inclusion of the $1$-skeleton 
\seqm{W_{2g}}{\iota}{\mf S_g}, and note $\iota\circ\alpha$ is nullhomotopic. If $f$ exists, $\bar f\circ \alpha$ is nullhomotopic since $\bar f=f\circ\iota$ by definition. 
Conversely, if $\bar f\circ \alpha$ is nullhomotopic, then we can extend $\bar f$ to a map $f$ by mapping the interior of the attached $2$-cell 
$\Int(D^2)\cong S^1\times(0,1]/S^1\times\{1\}$ at each $S^1\times\{t\}$ via this nullhomotopy.]

\item Since $\pi_2(\mc O_3)=0$, any two maps $\seqm{\mf S_g}{f}{\mc O_3}$ and $\seqm{\mf S_g}{f'}{\mc O_3}$ that agree on $W_{2g}$ must be homotopic
to each other. [Proof: we show that $f$ and $f'$ are homotopic along a homotopy that is fixed when restricted to $W_{2g}$, where they agree; we 
denote their common restriction $\seqm{W_{2g}}{\bar f}{\mc O_3}$. Consider the map $\seqm{S^2}{\omega}{\mc O_3}$ defined by mapping the equator 
$S^1\subset S^2$ via $\bar f\circ\alpha$, and mapping the interiors of its top and bottom $2$-disk hemispheres $D_a^2$ and $D_b^2$ (which are bounded by the equator) 
respectively via $f$ and $f'$ restricted to the interior of the $2$-cell of $\mf S_g$. 
Continuity is not broken since the boundaries of $D_a^2$ and $D_b^2$ map to $\mc O_3$ identically via $\bar f\circ\alpha$, which agrees with how the $2$-cell's boundary 
is mapped by $f$ and $f'$ after gluing to $W_{2g}$. Since $\pi_2(\mc O_3)=0$, $\omega$ is nullhomotopic, so there exists a map 
$\seqm{D^3}{\widehat\omega}{\mc O_3}$ that restricts to $\omega$ on the boundary of the $3$-ball $D^3$. Notice we can homotope $D_a^2$ to $D_b^2$ in $D^3$ 
via a homotopy of embeddings of disks that stays fixed on their common boundary along the equator. We can then define a homotopy of $f$ to $f'$ that is fixed everywhere 
except the interior of the $2$-cell (i.e. fixed on $W_{2g}$) by mapping the interiors of the $2$-disks in this homotopy of embeddings through $\widehat\omega$.]
\end{enumerate}

These two facts imply that the map of sets (not groups) \seqm{\vbr{\mf S_g,\mc O_3}}{\iota_*}{\vbr{W_{2g},\mc O_3}} given by $\iota_*([f]) := [f\circ\iota]$
is injective, so $\vbr{\mf S_g,\mc O_3}$ can be indentified with the image of $\iota_*$, and in turn, the image of $\iota_*$ can be identified with the classes in the 
group $\vbr{W_{2g},\mc O_3}$ that are mapped to identity by the map of sets \seqm{\vbr{W_{2g},\mc O_3}}{\alpha_*}{\vbr{S^1,\mc O_3}}; 
$\alpha_*([\bar f]) := [\bar f\circ\alpha]$. Namely, we have
\begin{align*}
\vbr{\mf S_g,\mc O_3} & = \cset{\mf k\in\vbr{W_{2g},\mc O_3}}{\alpha_*(\mf k)=1\in 2\mc D_3} \\ 
& = \cset{(\mf a_1,\mf b_1,\ldots,\mf a_g, \mf b_g)\in \bigoplus^{2g}  2\mc D_3} {[\mf a_1,\mf b_1]\cdots[\mf a_g,\mf b_g] = 1\in 2\mc D_3}
\end{align*}
since $\vbr{S^1,\mc O_3}=\pi_1(\mc O_3)\cong 2\mc D_3$ and  
$\vbr{W_{2g},\mc O_3}\cong \bigoplus^{2g} \pi_1(\mc O_3)\cong \bigoplus^{2g}  2\mc D_3$, with $\alpha_*$ given by mapping 
$(\mf a_1,\mf b_1,\ldots,\mf a_g, \mf b_g)\mapsto [\mf a_1,\mf b_1]\cdots[\mf a_g,\mf b_g]$. 

The main point here is that the based homotopy classes $\vbr{\mf S_g,\mc O_3}$ are a subset of $\vbr{W_{2g},\mc O_3}$ under the inclusion $\iota_*$, 
so the homotopy classes of maps \seqm{\mf S_g}{}{\mc O_3} are determined by their restriction to $W_{2g}$. Thus, if $\bd M_{\mc S}:=\mf S_g$, 
we only need to put constraints on the subcomplex of the $1$-skeleton $sk^1\bd M_{\mc S}$ whose embedding is homeomorphic to the embedding of
$W_{2g}$ in $\mf S_g$. By Remark~\ref{RTMain}, 
this means we can take $W_{2g}$ in place of $\bd M$ in Theorem~\ref{TMain}. We also do not need to have a frame field predefined outside $W_{2g}$, 
though in this case we have to ensure that the frame field $f\colon\seqm{W_{2g}}{}{\mc O_3}$ satisfies $\alpha_*([f])=1\in 2\mc D_3$ so that it can be 
extended over the rest of the boundary surface (let alone the interior).

\subsection{Constructing nullhomotopies and defiing the field}

Finally, part (ii) of the proof of Theorem~\ref{TMain} suggests an algorithm for constructing a frame field from given boundary and singularity 
constraints when the conditions of the theorem are met. There is still some work to be done before the construction of the frame field can be 
made fully explicit. Namely, the proof uses nullhomotopies of loops \seqm{S^1}{}{\mc O_3} mapped from the boundaries of $2$-cells in 
$sk^2 M_{\mc S}$ in order to extend the frame field $f$ to $2$-cells. This is followed by nullhomotopies of maps \seqm{S^2}{}{\mc O_3} 
mapped from the boundaries $3$-cells in $sk^3 M_{\mc S}$ to extend the frame field to $3$-cells. In both cases these nullhomotopies are only 
shown to exist, and an explicit construction is still needed to describe $f$ concretely. At the same time we have to choose representative loops
$\chi_j$ and $\gamma_e$ for the solutions $x_j$ and $y_e$ for the system, and then -- when all is said and done -- the resulting field might not be 
very smooth (e.g. it might vary rapidly across edges between adjacent vertices in $sk^1 M_{\mc S}$, depending on how they are embedded).

To begin, we can use the covering $\seqm{S^3}{r}{\mc O_3}$ to obtain explicit nullhomotopies of the loops \seqm{S^1}{\omega}{\mc O_3} on 
boundaries of $2$-cells. This follows the same process that was used to determine the classes $[\zeta(e)]$. As we did for $\zeta(e)$, we construct 
a lift $\lambda$ of $\omega$ in $S^3$, but this time we can expect $\lambda$ to be a loop itself since $\omega$ is nullhomotopic (both endpoints 
in the lift must be the same element in the $2\mc D_3$-fiber). In any case, a loop in $S^3$ such as $\lambda$ can easily be nullhomotoped to it's 
basepoint $b$ by first homotoping $\lambda$ slightly so that it avoids the antipode $-b$ everywhere, then pushing $\lambda$ towards $b$ along the 
geodesic lines from $-b$ to $b$. Mapping this nullhomotopy to $\mc O_3$ via $r$ defines a nullhomotopy of $\omega$.

Once the $2$-cells are filled we are left with finding nullhomotopies of the maps \seqm{S^2}{\kappa}{\mc O_3} from $3$-cell boundaries.
Since covering spaces have general homotopy lifting properties, not just path lifting, a covering space argument using the covering $r$ could
be helpful here as well. This time we might be able construct a lift \seqm{S^2}{\omega}{S^3} of $\kappa$ similarly as we did for loops 
(which must exist since $\kappa$ is nullhomotopic, and in fact $r$ induces the trivial isomorphism on $\pi_2()$), say by subdividing the sphere 
$S^2$ into sufficiently small closed disks. Which ever way our lift $\omega$ is obtained, it can be nullhomotoped is $S^3$ almost as easily as our
loop $\lambda$ was since it cannot be a surjective map to $S^3$ for dimensional reasions. An alternative approach works more generally, not just 
in the presence of covering spaces. Suppose we have a $CW$-decomposition (not necessarily regular) of the $3$-manifold $\mc O_3:=SO(3)/\mc D_3$,
as would follow by finding a Morse function
$$
g\,\colon\,\seqm{\mc O_3}{}{\mb R}.
$$
In this case $\kappa$ can easily be homotoped off of interiors of $3$-cells in $sk^3\mc O_3$ (since $\kappa$ cannot map there surjectively for 
dimensional reasons), for example, by picking any interior point in each $3$-cell $C$ where $\kappa$ does not map to, removing this point from 
$C$ and contracting the remainder to its boundary, in the process homotoping $\kappa$ to its boundary as well (or more concretely, by pushing 
$\kappa$ away from corresponding critical points of $g$ and following its gradient flow). This is a homotopy of $\kappa$ into the $2$-skeleton 
$sk^2\mc O_3$, and while the resulting map might be surjective onto the interiors of $2$-cells this time, it will aslo likely be highly folded. 
Unfolding it might at least make it non-surjective on interiors of $2$-cells so that it can be homotoped as before into the $1$-skeleton where 
the problem is easier. The map will still be folded in the $1$-skeleton however, which means there is a possibility that continuing the unfolding 
here might lead all the way to a nullhomotopy since a map of a $2$-sphere into the $1$-skeleton has to be nullhomotopic (since higher dimensional 
homotopy groups of graphs are trivial). A good optimization based smoothing or unfolding algorithm might be helpful here. 

These nullhomotopies of course depend on how the representative loops for the solutions $x_j$ and $y_e$ are chosen. They can be chosen
using the covering space constructions used to determine the boundary classes $[\zeta(e)]$.
As for smoothness, notice the definition of the frame field in part (ii) Theorem~\ref{TMain} has certain edges $e$ in $sk^1 M_{\mc S}$ tracing loops in 
$\mc O_3$ that go back-and-forth from the basepoint $f(b_0)$ (in order to to phrase the homotopy extension problem of $f$ over $2$-cells in terms of 
homotopy classes of $\pi_1(\mc O_3)$ based at $f(b_0)$). This in practice makes the frame field vary quickly across these edges. 
To remedy this, paths tracing to $f(b_0)$ and back should be contracted and the field smoothed continuously over $e$ before constructing the nullhomotopes 
above.

\subsection{Alternative decompositions}

There might be simpler algebraic systems for extending frame fields if we consider decompositions of $M_{\mc S}$ besides regular $CW$-decompositions.
For example, if we relax our $CW$-decomposition being regular we may obtain a considerable smaller cell count, especially if we consider these
decompositions as being only up-to-homotopy. Namely, there is a space $X$ with the simplest possible cell structure that is homotopy equivalent to $M_{\mc S}$.
Such spaces might be constructed with Discrete Morse theory -- though whichever way we construct them, some work would have to be done reversing the
homotopy equivalence in order to construct our frame field on $M\sm \mc S$.

On the other hand, since we are dealing with $3$-manifolds, decompositions that are naturally suited towards them  such as handle decompositions might be simpler.
In certain cases $3$-manifolds can be decomposed and simplified by introducing cuts along handles. For example, a handle body $M_{\mc S}$ can be simplified to a 
$3$-ball by iteratively cutting along handles to produce a sequence of submanifolds $M_{\mc S}=:H_0, H_1,\ldots,H_k\cong D^3$. Formally, at each cut $j$ we are 
embedding a $2$-disk $D_j$ whose circle boundary is in $\bd H_j$ and not bounded by a disk in $\bd H_j$, and whose interior is in the interior of $H_j$. 
The $2$-disk is fattened on both sides to $F_j:=D_j\times [-1,1]$ so that $\bd D_j\times [-1,1]$ and $\Int(D_j)\times [-1,1]$ remain on the boundary and interior of 
$H_j$ respectively, then $D_j\times (-1,1)$ is deleted from $H_j$ to obtain $H_{j+1}$ with a simpler boundary. Notice that a $3$-frame field $f$ that is already defined on 
$\bd M_{\mc S}$ (as a map to $\mc O_3$) extends over all of $M_{\mc S}$ if and only if $f$ restricts to a nullhomotopic map $\seqm{\bd D_j\cong S^1}{}{\mc O_3}$ 
on the circle boundary of each cut. To see this, notice first that this is necessary and sufficient for $f$ to extend over each $F_j$  since $F_j$ deformation retracts 
onto $G_j:=(\bd D_j\times [-1,1])\cup (D_j\times\{0\})$ (nullhomotopy extends the field from $(S^1\times\{0\})$ to $(D_j\times\{0\})$ so that it is defined on $G_j$, 
then a retraction $\seqm{F_j}{r}{G_j}$ defines a field $f\circ r$ on $F_j$). Outside the $F_j$'s all we have is a $3$-ball $H_k\cong D^3$ remaining, and once $f$ is defined 
on each $F_j$ (on top of $\bd M_{\mc S}$), then it is defined on the $S^2$ boundary of $H_k$ as well. All that is left is to extend it into the interior of $H_k$, which can be 
done since $\pi_2(\mc O_3)=0$.

\nocite{MimuraToda, Hatcher, MosherTangora, Munkres, Milnor, McCleary, CoxeterMoser, Armstrong2015, Vaxman2017}

\bibliographystyle{amsplain}
\bibliography{mybibliography}

\end{document}